\documentclass[12 pt, reqno]{amsart}
\usepackage{reeb, fullpage}

\begin{document}

\title{Zig--Zag Modules: Cosheaves and  K-Theory}

\author{Ryan Grady
\and Anna Schenfisch}

\address{Department of Mathematical Sciences\\Montana State University\\Bozeman, MT 59717}
\email{ryan.grady1@montana.edu}

\address{Department of Mathematical Sciences\\Montana State University\\Bozeman, MT 59717}
\email{annaschenfisch@montana.edu}

\thanks{AS is supported by the National Science Foundation under NIH/NSF DMS 1664858.}

\begin{abstract}
Persistence modules have a natural home in the setting of stratified spaces and constructible cosheaves. In this article, we first give explicit constructible cosheaves for common data-motivated persistence modules, namely, for modules that arise from zig-zag filtrations (including monotone filtrations), and for augmented persistence modules (which encode the data of instantaneous events). We then identify an equivalence of categories between a particular notion of zig-zag modules and the combinatorial entrance path category on stratified $\mathbb{R}$. Finally, we compute the algebraic $K$-theory of generalized zig-zag modules and describe connections to both Euler curves and $K_0$ of the monoid of persistence diagrams as described by Bubenik and Elchesen.
\end{abstract}


\maketitle

\tableofcontents

\maketitle


\section{Introduction}

In this article we aim to demonstrate the utility of viewing persistent phenomena from the perspective of constructible (co)sheaves. In particular, we demonstrate how cosheaves provide a convenient interpretation of augmented descriptors of persistence modules and how cosheaves are a convenient setting for constructing invariants via algebraic K-theory.  The present is in the same spirit of the program we first employed in \cite{GS}, namely, applying stratified mathematics and higher algebra to topological data analysis (TDA).

The use of cosheaves in TDA goes back at least to Curry \cite{curry2015}. The work of Curry and collaborators (e.g., work with Patel \cite{CP}), serves as an inspiration for our own perspectives. The key idea interpolating between persistence modules and constructible cosheaves is that of a stratified space. A persistence module $\{V_i\}_{i \in I}$ is obtained by sampling (or otherwise selecting a discrete subset of) a larger parameter space. For concreteness, consider $I \subset \RR_{\ge 0}$ as a selection of ``instances" in our one-dimensional ray of ``time." As our persistence module only changes at elements of $I$, it is locally constant on $\RR \setminus I$, which is the defining property of a constructible cosheaf.

Constructible cosheaves are particularly nice mathematical objects for several reasons, chief among them is their equivalence to representations of the so called \emph{entrance path category}; this is known as ``the" Exodromy Theorem. Any stratified space has an associated entrance path category and in good cases (e.g., when the space is a combinatorial manifold), the entrance path category is a straightforward combinatorial object -- in many cases it's simply a poset. The idea of \emph{exodromy} goes back to MacPherson and proofs in different settings appear in work of Curry and Patel \cite{CP}, Treumann \cite{Treumann}, Lurie \cite{lurie}, and Barwick with Glasman and Haine \cite{Barwick}.

Given a parameter space (and a choice of sampling instances), exodromy allows us to consider all persistence modules/constructible sheaves on that space as a category of functors. Such categories of functors then inherit desirable properties from the target category. For instance, if we consider modules valued in vector spaces, the category of persistence modules is naturally an Abelian category. Abelian categories are the home of homological and homotopical algebra, so we are free to apply the tools of algebraic topology/homotopy theory, e.g., algebraic K-theory. The combinatorial nature of the entrance path category makes K-theory computations tractable and allows us to consider connections with other persistent invariants such as Euler curves and persistence diagrams.

In the present article, we are mainly concerned with one-dimensional parameter spaces. The resulting persistence modules are  the zig-zag persistence modules of Carlsson and de Silva \cite{carlsson2010zigzag}, which includes the more typically seen monotone (standard) modules.

Readers familiar with the persistent homology transform (PHT) may be interested to note that the PHT is a special type of persistence module itself. Our computation of $K$-theory for zig-zag persistence modules has an interpretation in the setting of the PHT where the sphere of directions is $S^1$. Thus, the results of this paper may be useful for future work in the computation of other invariants of the PHT. See \cite{turner} for further background on the PHT.

\subsection{Why K-theory?}

In the later part of this article, we compute the K-theory of the category of zig-zag modules. Here, we briefly overview why K-theory is a useful invariant.

K-theory began as simply as group completion of a monoid. Indeed, let $(M,\oplus)$ be a commutative monoid and define $K_0 (M, \oplus)$ to be the unique (up to isomorphism) Abelian group, equipped with a monoid homomorphism from $M$, satisfying the universal property: for any Abelian group $A$ and homomorphism (of monoids) $\varphi \colon M \to A$, there exists a unique group homomorphism factorization through $K_0 (M, \oplus)$. This universal property is described as the \emph{universal Euler characteristic} and is conveyed diagrammatically as follows:
\[
\xymatrix{ (M, \oplus) \ar[d] \ar[r]^{\quad \forall} & A \\ K_0 (M, \oplus) \ar@{-->}[ur]_{\exists !} }.
\]
For instance, let $\cV$ be the isomorphism classes of finite dimensional vector spaces over the field~$\mathbb{R}$ (with direct sum) and~$\varphi : \cV \to \mathbb{Z}$ the rank function, then we have an induced map~$K_0 (\cV) \to \mathbb{Z}$, which happens to be an isomorphism. Expanding this example to complexes, let $\cC$ denote isomorphism classes of bounded complexes of $\mathbb{R}$-vector spaces. The natural extension of the rank function is the Euler characteristic, which again factors uniquely through $K_0 (\cC)$. (In the topological setting, the Chern character of a vector bundle is an example of such an additive map.) The universal property of $K_0$ extends to categories equipped with a symmetric monoidal structure, as isomorphism classes of objects in such a category form a commutative monoid.

K-theory is more than just a single Abelian group, but rather a spectrum, $\KK(\cC)$, associated to a category (equipped with additional structure).  Recall that spectra are the central objects of homotopy theory. The homotopy groups of $\KK(\cC)$ define the K-groups of $\cC$, i.e., $K_n(\cC) := \pi_n (\KK(\cC))$.  To first approximation, spectra can be thought of as the objects that parametrize cohomology theories. As such, they, so K-theory in particular, admit a wealth of computational tools, refined structures, and interpretations from algebraic topology. Cohomology theories are also the natural home for obstruction/anomaly theory and in this way, K-theory has become a central tool in topology (index theory, finiteness obstructions) and algebraic number theory (class field theory).

When refined to the level of spectra, K-theory has a remarkable additive structure with respect to split short exact sequences of categories. (We discuss this in Appendix \ref{app:A}, see also \cite{BGT}.) Combined with its property as the universal Euler characteristic, K-theory is the {\it universal additive invariant} of (Waldhausen) categories.


\subsubsection{Flavors and History of K-theory}

There are several flavors and constructions of K-theory; we trace here the history to the two we use in the present article: Waldhausen's construction and Zakharevich's theory of assemblers. See the canonical texts of  Rosenberg \cite{Rosenberg} and Weibel \cite{Weibel} for historical references and more details on the development of K-theory. 

The genesis of K-theory came in the late 1950's and early 1960's through the work of Grothendieck in complex (algebraic) geometry and Atiyah and Hirzebruch in topology.  Algebraic K-theory---the kind relevant to the present work---is an extension of Grothendieck's ideas to build a family of functors from rings to Abelian groups $K_i : \mathsf{Ring} \to \mathsf{Ab}$. While Grothendieck only defined $K_0$, suitable definitions for $K_1$ and $K_2$ were found by the mid 1960's; the contributions of Bass, Schanuel, and Milnor are most notable. (Bass and Karoubi also gave definitions of negative K-theory, $K_{-n} (R)$.) Definitions of higher K-groups was a major open problem in the early 1970's, which was first solved by Dan Quillen: the $+$-construction. (Milnor had given a definition of higher K-groups as well, though this \emph{Milnor} K-theory is only a summand of the now accepted definition of higher K-theory.) Given a ring, $R$, Quillen produced a space, $BGL(R)^+$, whose homotopy groups recovered/defined the K-theory of $R$.

Quillen quickly followed his $+$-construction with the $Q$-construction. The $Q$-construction takes as input an exact category, $\cC$, e.g., the category of finitely generated projective modules for a ring, and outputs a space, $\Omega B Q \cC$, whose homotopy groups define K-theory. Quillen used the $Q$-construction to prove many fundamental results in algebraic K-theory that restricted to those for rings, as he also proved that $+=Q$, that is, the $Q$-construction is a strict generalization of $+$-construction for rings.

The next revolution in algebraic K-theory came through Waldhausen's work in manifold topology \cite{Wald}. Published in 1985, Waldhausen gave a construction that takes as input categories with structure that generalizes that of exact categories---nowadays called Waldhausen categories---and outputs a spectrum (the basic building block of homotopy theory) whose homotopy groups define the corresponding K-groups. (Segal some 15 years earlier used his~$\Gamma$-objects to produce a K-theory spectrum in certain cases.) Waldhausen's construction is often referred to as the $S_\bullet$-construction and we give a brief overview in Appendix \ref{app:A}. The~$S_\bullet$-construction is a strict extension of the $Q$-construction. Perhaps most significantly, the~$S_\bullet$-construction satisfies an additivity result for split short exact sequences; this result has become a central tool in algebraic K-theory.

Finally, we note that there has been an extension of K-theory to the higher categorical/homotopical algebraic setting as well. The work of Blumberg, Gepner, and Tabuada \cite{BGT} proves that K-theory satisfies certain universal properties, such as additivity, (and hence is essentially uniquely defined by such properties) in this setting.

\subsubsection{K-theory and persistence}

Through the work of Patel, Bubenik and collaborators, K-theoretic considerations have started to appear in the TDA literature. Patel considered the Grothendick group, i.e., $K_0$, of one-dimensional persistence modules valued in symmetric monoidal categories \cite{PatelGen}.

Subsequently, in \cite{BubHom}, Bubenick and Mili{\'c}evi{\'c} show that the category of persistence modules over any preorder is Abelian. The key idea---which we use below as well---is that functor categories inherit many of the properties of the target category, so if the target is Abelian or Grothendieck, i.e., AB5 with a generator, then the functor category with domain a preorder (or any small category) is Abelian or Grothendieck. It would be interesting to apply Quillen's $Q$-construction to these categories of persistence modules and compare the resulting K-theories to our computations below. (We note that \cite{BubHom} contains much more than we just outlined, e.g., the authors prove an embedding theorem in the vein of the Gabriel--Popescu Theorem.)

More relevant for us is the recent article \cite{BubVirtual} by Bubenik and Elchesen. In this work, the group completion of the monoid of persistence diagrams is described, i.e., $K_0 (\mathsf{Diag})$ is defined (semi-)explicitly. Points in diagrams are counted with multiplicity, so the binary operation is simply induced by $+ \colon \NN_0 \times \NN_0 \to \NN_0$. The input data for the construction of Bubenik and Elchesen is pretty flexible, so one can talk about diagrams (and their group completions) indexed by the entire first quadrant, the integers, etc. We make contact with this work in Section \ref{sect:virtual} below.

\subsection{What we do}

We have aimed to illustrate the connection between persistence modules and cosheaves and the utility of this interplay. To this end, we accomplish the following.

\subsubsection{coSheaves from filtrations} 

The relevance of cosheaves in TDA has been advocated by Curry and others for a number of years. In Section \ref{sect:cosheaves}, we give explicit constructions of constructible cosheaves associated to persistence modules. We are particularly interested in persistence modules arising from index filtrations of spaces. In Section \ref{sect:aug}, we describe the \emph{augmented filtration cosheaf}, which records both non-instantaneous and instantaneous events. (We flag the recent work of Berkouk, Ginot, and Oudot \cite{BGO} where level-set persistence is recast in terms of sheaves over $\RR$.)

\subsubsection{Equivalence Theorem}

We prove an equivalence of categories between a localization of the category of zig-zag modules a la Carlsson and de Silva and constructible cosheaves on $\RR$. The explicit statement of the result is Theorem \ref{thm:equivcat}.  This result is stated in passing (Example 6.3) in the recent work of Curry and Patel \cite{CP}, and we make it explicit with proof. We hope the proof is as interesting to the reader as the result, though it uses techniques that are different from the rest of the paper so we relegate it to Appendix \ref{app:B}.

One motivation for this result is to argue that our $K$-theoretic computations which follow deserve to be called {\em the} $K$-theory of zig-zag modules.

\subsubsection{K-theory of Zig-Zag Modules}

In Section \ref{sect:ktheory}, we define and compute K-theory of persistence modules, viewed as constructible cosheaves on a stratified parameter space.   We use Waldhausen's $S_\bullet$ construction of K-theory.  A key input is \emph{additivity}, in this case with respect to strata. For instance, 
in the case that our parameter space is one-dimensional, e.g., monotone or zig-zag persistence, the group $K_0$ is the free abelian group on the strata of parameter space (Theorems \ref{thm:bigOplus} and \ref{thm:set}). This result is true for both $\Vect$ valued modules and modules valued in pointed sets.

The higher K-groups do not vanish but rather are given by the algebraic K-theory of fields and/or the sphere spectrum. In forthcoming work, we aim to interpret these higher K-groups as arising from data.

The constructions and techniques we present apply to parameter spaces of arbitrary dimension.

\subsubsection{Euler Curves and Virtual Diagrams}

We conclude the body of the paper by connecting our K-theoretic work back to some recent work in TDA. First, we show how the Euler curve of a persistence module has a natural interpretation as a class in K-theory. (This is as expected, e.g., Kashiwara and Schapira \cite{KS} prove that $K_0$ is isomorphic to constructible functions via a local Euler index.)
With this observation, we define an \emph{Euler class} for arbitrary persistence modules regardless of dimension; this is Definition \ref{defn:euler}.  Lastly, Section \ref{sect:virtual} builds a group homomorphism from $K_0$ of persistence modules to Bubenik and Elchesen's Abelian group of virtual persistence diagrams.

\subsection*{Conventions}
We assume the reader has some familiarity with algebraic topology, and freely use concepts from Hatcher's standard text~\cite{hatcher}. 

Throughout, we will let $\Vect_\FF$ be the category of finite dimensional vector spaces over the field $\FF$ and linear maps. Much of our work doesn't depend on making a choice of field and we simply use the notation $\Vect$.

Unless otherwise noted, we will assume all stratified spaces are combinatorial manifolds equipped with their native stratification, notions we define in the next section.

\subsection*{Acknowledgements}
We thank David Ayala for many discussions related to the content of this and other articles. We also thank Peter Bubenik for several discussions related to zig-zag persistence and other mathematical topics in TDA. Finally, we thank the anonymous referee for feedback and suggestions which have greatly enhanced the readability of the manuscript.

\section{Constructible coSheaves}\label{sect:prelims}

This section is a terse introduction to terminology and notation we will use throughout the sequel. Examples and further details are abundantly available, e.g., \cite{CP} or \cite{GS}.

\subsection{Stratified/Constructible Basics}

\begin{definition}
Let $(\cP, \le)$ be a poset. The {\em upward closed topology} on $(\cP,\le)$ is defined as follows: $U \subseteq \cP$ is open if and only if for all $u \in U$, $\cP_{u \le} := \{ p \in \cP ~|~ u \le p\} \subseteq U.$ 
\end{definition}

The upward closed topology is also known as the {\em Alexandrov} topology.

\begin{definition}
A {\em stratified topological space} is a triple $(X \xrightarrow{\phi} \cP)$ consisting of
\begin{itemize}
    \item a paracompact, Hausdorff topological space, $X$,
    \item a poset $\cP$, equipped with the upward closed topology, and
    \item a continuous map $X \xrightarrow{\phi} \cP$.
\end{itemize}
\end{definition}

Note that any topological space is stratified by the terminal poset consisting of a singleton set. Moreover, the simplices of a simplicial complex, $K$, come equipped with the structure of a poset, and we call the resulting stratification of $K$ the \emph{native stratification} which will denote by $\Nat(K)$.

\begin{definition} Given a stratified topological space $\phi : X \to \cP$, and any $p \in \cP$, the {\em $p$-stratum}, 
$X_p$, is defined as
$$X_p := \phi^{-1}(p).$$
\end{definition}

\begin{example}\label{ex:s1}
For $n \in \NN$, let $[n]$ denote the totally ordered set $\{0<1<\dotsb<n\}$. Consider a stratified circle, $S^1 \to [1]$, stratified by $v$, a single vertex, and $\alpha$, the arc which is the complement of $v$. This example is illustrated in Figure \ref{fig:s1}. (So the map $\phi : S^1 \to [1]$ is given by $v \mapsto 0$ and $\alpha:= S^1\setminus \{v\} \mapsto 1$.) The 0-stratum is the vertex $v$ and the $1$-stratum is the arc $\alpha$, i.e., $S^1_0 = \{v\}$ and $S^1_1 = S^1 \setminus \{v\}= \alpha$.
\end{example}

\begin{figure}[h!] \label{fig:s1}
\centering
\includegraphics[width=.25\textwidth]{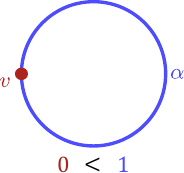}
\begin{caption}{A stratified circle, $S^1 \to [1]$ as in Example \ref{ex:s1}, where $v \mapsto 0$ and $\alpha \mapsto 1$.}
\end{caption}
\end{figure}


\begin{definition} 
A {\em map of stratified topological spaces} $(\phi \colon X \to \cP)$ to $(\psi \colon Y \to \cQ)$ is a pair of 
continuous maps $(f_1, f_2)$ making the following diagram commute.
\[
\xymatrix{
X \ar[r]^{f_1} \ar[d]_{\phi} & Y \ar[d]^{\psi}\\
\cP \ar[r]^{f_2} & \cQ.}
\]
A map of stratified spaces is a {\em stratified homeomorphism} if it admits a two-sided (stratified) inverse.

\end{definition}

\begin{definition}
Let $X \to \cP$ be a stratified space and $x \in X_p \subseteq X$ a point. The space $X$ is \emph{conically stratified at} $x$ if there exists an open neighborhood, $U_x$, of $x$ and a stratified homeomorphism $U_x \cong Z \times CY$ where $Z$ is a topological space and $CY$ is the cone on a space $Y$ stratified by $\cP_{>p}$. A space is conically stratified if it is conically stratified at all of its points.
\end{definition}

\begin{definition}
Let $L$ be a polyhedron, so every point admits a neighborhood which is a finite union of simplices. Recall that a map $f \colon L \to \RR^n$ is \em{piecewise linear} (PL) if there exists a triangulation of $L$ such that restricted to each simplex $f$ is linear.
\end{definition}

\begin{definition}
A \emph{piecewise linear (PL) manifold} is a topological manifold which admits an atlas where transition functions are piecewise linear\footnote{Admiting a PL atlas is equivalent to specifying a class of trangulations of the underlying manifold which is stable under subdivision.}.
\end{definition}

Completely analogously to smooth manifolds, PL manifolds form a category with morphisms being PL maps and isomorphisms being PL homeomorphisms.

\begin{definition} 
A \emph{combinatorial manifold} $X$ is a triangulated PL manifold. That is, a combinatorial manifold is a PL manifold $X$ along with a simplicial complex $K$ and a PL homeomorphism $K \to X$. The manifold $X$ inherits a native stratification from the simplicial complex $K$.
\end{definition}

\begin{remark}
As discussed in \cite{AFT}, every Whitney stratified manifold is conically stratified. In particular, a combinatorial manifold $X$ is conically stratified.
\end{remark}

For further details on PL and combinatorial manifolds, see \cite{RSPL} or Section 3.9 of \cite{thurston}.

\begin{definition}[\cite{GS}]\label{def:connectedstrat}
Let  $(S \xrightarrow{ \varphi } \cQ)$ be a stratified space, $S \hookrightarrow X$ a topological embedding, and $\pi_0 (X \setminus S)= \cA$. Define the poset, $\cQ^{\wwedge}$, as the set $\cQ \amalg \cA$, subject to the following generating relations:
\begin{enumerate}
\item The relations of $\cQ$;
\item For $\ell \in \cQ$ and $\alpha \in \cA$, $\ell \le \alpha$ if and only if $\varphi^{-1}_S (\ell) \subseteq \overline{\alpha}$, i.e., the $\ell$-stratum is in the closure of the connected component indexed by $\alpha$.
\end{enumerate}
There is an obvious extension of the map $\varphi$, $\psi_S \colon X \to \cQ^{\wwedge}$ and we call this stratification the {\em connected ambient stratification}. We often denote this stratified space by $(X,S)^{\wwedge}$.
\end{definition}

A typical (easy) example of the preceding is considering a discrete subset $I \subset \RR$. The resulting stratified space, $(\RR, I)^{\wwedge}$, is a combinatorial manifold.

\begin{definition}
Let $X$ be a topological space, $\sf{Op}(X)$ the poset of open sets in $X$, and $V$ a category. A \emph{precosheaf on $X$ valued in $V$} is a functor $\cF \colon \sf{Op}(X) \to V$. A precosheaf is a cosheaf if for each open $U \subseteq X$ and any open cover of $U$, $\{U_i \to U\}$, there is an equivalence (in $V$)
\[
\colim \left ( \coprod_{i,j} \cF(U_i \cap U_j )\rightrightarrows \coprod_i \cF(U_i) \right ) \xrightarrow{\sim} \cF(U).
\]
\end{definition}

For what remains, we will assume $V$ is a nice category, so that cosheafification exists. (Cosheafification is quite subtle, even compared to its dual notion of sheafification.) In particular, we will later focus on the case that $V=\mathsf{Set}$ or $V=\Vect$. 

\begin{lemma}\label{lem:bases}
Let $\cB$ be a basis for the topology of the space $X$ and let $\cF$ be a cosheaf on the poset determined by $\cB$. There is a unique (up to isomorphism) extension of $\cF$ to a cosheaf on $X$.
\end{lemma}

The idea of the lemma can be thought of in terms of a Kan extension diagram:
\[
\xymatrix{
\cB \ar[r]^{\cF} \ar@{^{(}->}[d] & V\\ \sf{Op}(X) \ar@{-->}[ur]_{\exists !}}.
\]

\begin{definition}
Let $M \to \cP$ be a stratified space (not nec. conical or simplicial) and $\cF$ a cosheaf on $M$. The cosheaf, $\cF$, is \emph{constructible} if it is locally constant when restricted to any stratum of $M \to \cP$, i.e., given $p \in \cP$ and $x \in M_p$ there exists a neighborhood $p \in U \subseteq M_p$ such that $\cF |_U$ is constant.
\end{definition}

\begin{definition}
Let $\cF$ be a (pre)cosheaf on $X$ and $p \in X$. The \emph{costalk of $\cF$ at $p$} is defined by
\[
\cF_p := \lim_{U \ni p} \cF (U) .
\]
\end{definition}

%

\subsection{Operations on coSheaves}

Given a continuous map $\xi \colon X \to Y$, there is in an induced functor on the posets of opens $\hat{\xi} \colon \sf{Op}(Y) \to \sf{Op}(X)$ given by preimages with respect to $\xi$.

\begin{definition}
Let $\xi \colon X \to Y$ be a continuous map and $\cF$ a (pre)cosheaf on $X$. The \emph{pushforward} of $\cF$, $\xi_\ast \cF$, is the (pre)cosheaf on $Y$ given by $\xi_\ast \cF := \cF \circ \hat{\xi}$.
\end{definition}

There is a contravariant functor as well associated to a map $\xi \colon X \to Y$: the pullback $\xi^\ast \colon \sf{coShv}(Y) \to \sf{coShv}(X)$. As a continuous map is not necessarily an open map, $\xi^\ast$ is (slightly) more involved to define: it is the limit over opens containing $\xi (U)$ for~$U \subseteq X$ an open. Only pushforwards appear below.

\begin{example}
Let $p \in X$ be a point in the topological space $X$ and $i \colon p \hookrightarrow X$ the inclusion map. Let $\cF$ be a cosheaf on $X$, then $i^\ast \cF$ is the costalk at $p$ of $\cF$, $\cF_p$. Let $W$ be a cosheaf on $p$, then $i_\ast W$ is a \emph{skyscraper cosheaf} on $X$.
\end{example}

\begin{example}\label{ex:Cmap}
Let $\xi \colon (\phi \colon X \to \cP) \to (\psi \colon Y \to \cQ)$ be a stratified map and $\cF$ a constructible cosheaf on $X$. Although the pullback of a constructible cosheaf is always constructible, it is not necessarily the case that $\xi_\ast \cF$ is constructible on $Y$.
\begin{itemize}
\item Consider the inclusion $\iota \colon [0,1/2) \hookrightarrow [0,1)$ and let $\underline{V}$ be a nonzero constant cosheaf on $[0,1/2)$. Further stratify $[0,1/2)$ and $[0,1)$ with zero-stratum $\{0\}$ and one-stratum $(0,1/2)$ (resp.~$(0,1)$). The cosheaf $\iota_\ast \underline{V}$ is not locally constant on $(0,1)$ as
\[
\iota_\ast \underline{V} (0,1/4) = V, \quad \text{ while } \quad \iota_\ast \underline{V} (3/4,1) = 0.
\]
\item Constructibility is preserved by pushforwards in certain cases. Let $C \colon [0,4] \to [0,3]$ be the ``elementary collapse" of the interval $[2,3]$, i.e.,
\[
C(t) = \begin{cases} 
t, \quad \;\; \text{ if } 0 \le t <2\\
2, \quad \; \; \text{ if } 2 \le t \le 3\\
t-1, \text{ if } 3<t\le 4
\end{cases}.
\]
The map $C$ is a stratified map with respect to the (connected) ambient stratifications induced by $\{0,1,2,3,4\} \subset [0,4]$ and $\{0,1,2,3\} \subset [0,3].$ Let $\cF$ be any constructible cosheaf on $[0,4]$. It is straightforward to verify that $C_\ast \cF$ is constructible on $[0,3]$.
\end{itemize}

\end{example}

\subsection{Entrance Paths and Their Representations}

Given a stratified space, $M \to \cP$, an entrance path is a continuous path in $M$ such that it for all time it stays in a stratum or enters into a deeper (with respect to $\cP$) stratum.

\begin{definition}
Let $M \to \cP$ be a stratified space. The \emph{entrance path category} of $M \to \cP$, $\sf{Ent}(M,\cP)$ has objects the points of $M$ and morphisms (elementary) homotopy classes of entrance paths.
\end{definition}

\begin{exodromy}[Theorem 6.1 of \cite{CP}]
Let $M \to \cP$ be a conically stratified space and $V$ a category. There is an equivalence of categories
\[
\sf{cShv}^V_{cbl}(M,\cP) \xrightarrow{\sim} \sf{Fun} (\sf{Ent}(M,\cP), V)
\]
between constructible cosheaves on $M$ and representations of its entrance path category.

\end{exodromy}

\begin{definition}
Let $M \to \cP$ be a combinatorial manifold. The \emph{combinatorial entrance path category}, $\sf{Ent}_\Delta (M,\cP)$ has as objects the strata of $M$ and a morphism $\sigma \to \tau$ whenever $\tau$ is a face of $\sigma$.

\end{definition}

\begin{proposition}
Let $M \to \cP$ be a combinatorial manifold. There is an equivalence of categories $\sf{Ent} (M,\cP) \xrightarrow{\sim} \sf{Ent}_\Delta (M,\cP)$. 
\end{proposition}

\begin{proof}
Define a functor $F: \sf{Ent} (M,\cP) \to \sf{Ent}_\Delta (M,\cP)$, where the image of a point $x \in \sf{Ob}(\sf{Ent} (M,\cP))$ is unique simplex $\sigma$ containing $x$, and the image a morphism $x \to y$ is the combinatorial entrance path from $F(x) \to F(y)$ (well-defined since the simplex containing $y$ must be a face of the simplex containing $x$ in order to be an entrance path). We claim that $F$ is fully faithful and essentially surjective. Again, let $x \in \sigma$,~$y \in \tau$, so that $\tau$ is a face of $\sigma$. Since $\tau$ and $\sigma$ are face-coface pairs of a non-degenerate triangluation, the subspace $\tau \cup \sigma$ deformation retracts onto $\tau$, meaning there is a unique homotopy class of entrance paths $x \to y$. Furthermore, there is a unique morphism $\sigma \to \tau$ in $\sf{Ent}_\Delta (M,\cP)$, i.e., $F$ is fully faithful. Next, we observe that, for any simplex $\sigma \in \sf{Ob}(\sf{Ent}_\Delta (M,\cP))$, we can always find a point $x$ so that $F(x) = \sigma$ (for example, let $x$ be the barycenter of $\sigma$). That means we have shown $F$ is also essentially surjective, and thus gives an equivalence of categories.
\end{proof}

\begin{remark}
One might hope that there is an equivalence of entrance and combinatorial entrance path categories for a larger class of stratifications. However, even when a space is stratified by a ``degenerate'' triangulation, this equivalence does not generally hold.
For instance, consider the stratified space shown in Figure \ref{fig:s1}, $S^1 \to [1]$, stratified by $v$, a single vertex, and its complement $\alpha$, an open arc. Let $x \in \alpha$. Then there are two distinct homotopy classes of entrance paths from $x \to v$ in $\sf{Ent} (M,\cP)$, but only one combinatorial entrance path given by the face relation in $\sf{Ent}_\Delta(M,\cP)$.
\end{remark}

One useful interpretation of the preceding proposition is that the data of a constructible cosheaf on a combinatorial manifold is just a specification of costalks on the stratifying poset and linear maps between them.

\section{Persistence Modules, Persistence Cosheaves, and Filtrations}\label{sect:cosheaves}

We now introduce our main actors: persistence modules and persistent cosheaves. To start, we consider constructible cosheaves that arise from common types of persistence modules and/or filtrations. We construct these cosheaves in a way that is compatible with traditional models of the specific filtration or module in question. We finish the section with an equivalence result relating zig-zag modules to one-dimensional constructible cosheaves.

\subsection{Persistent Definitions}

\begin{definition}\label{def:module}
A \emph{persistence module} is a functor $P \colon \mathbb{P} \to \Omega$, where $\mathbb{P}$ is some poset category. Specifically, we may refer to these as \emph{$\mathbb{P}$-indexed persistence modules.}  $\mathbb{P}$-indexed persistence modules define a category: the functor category,  whose morphisms are natural transformations between the functors. 
\end{definition}

Hereafter, we take $\Omega$ to be $\Vect_\FF$, the category of vector spaces over a field $\FF$, and by $\FF^i$, we mean a vector space of dimension $i$ in $\Vect_\FF$.

The previous definition is general -- in what follows, we will mostly restrict our attention to single-parameter persistence modules. There are two flavors of such modules common in the literature: zig-zag \cite{carlsson2010zigzag} and monotone (standard) persistence modules \cite{zomorodian2005computing}. We note that monotone persistence modules are most commonly called simply `persistence modules;' we have added the word `monotone' to emphasize their distinction from more general modules.

To any poset there is an associated undirected graph: its Hasse diagram. Properties of the Hasse diagram, e.g., if a Hasse diagram is planar, are used in order theory as they are often more accessible than the abstract poset. We will find it useful to consider the Hasse diagram of a poset as a one-dimensional simplicial complex.

\begin{definition}
Let $\PP$ be a poset.
\begin{itemize}
\item The poset $\PP$ is  \emph{zig-zag} if its Hasse diagram is homeomorphic to the closed interval, half-closed interval, or $\RR$.
\item A representation of a zig-zag poset $P \colon \mathbb{P} \to \Vect$ is a \emph{zig-zag persistence module}.
\item If $\PP$ is a linear order, then a representation $P \colon \mathbb{P} \to \Vect$ is a \emph{monotone persistence module}.
\end{itemize}
\end{definition}

Consider a zig-zag persistence module $P \colon I \to \Vect$, where the objects of $I$ are a discrete subset of real numbers (with potentially non-standard ordering). This, in turn, defines a stratification of $\RR$, the connected ambient stratification $(\RR, I)^{\wwedge}$ (where we have a zero-stratum for every object of $I$ and a one-stratum for every connected component of $\RR \setminus I$, as in~\cite{GS}). To define a cosheaf on $\RR$, it suffices to define its values on a basis of the topology on $\RR$.

First, we give a cosheaf theoretic interpretation of the notion of zig-zag modules found in~\cite{carlsson2010zigzag}. We call this cosheaf \emph{propagated} because the functor is entirely determined by the ordering of and assignments to zero strata; the value over a one-stratum is propagated from either endpoint depending on the ordering of the relevant poset.

\begin{construction}[The Propagated Persistence Cosheaf on $\RR$]
Given $P: I \to \Vect$ with $I \subset \RR$ discrete, we define the \emph{the propagated persistence cosheaf} $F_P \colon \sf{Opens}(\RR) \to \Vect$~as follows. Let $B_\epsilon \subset \RR$ be a metric $\epsilon$-ball so that $2\epsilon > 0$ is less than the distance between any pair of zero-strata. Then $B_\epsilon$ either contains a single zero-stratum or no zero-stratum, and we assign values for $F_P$ as follows:
\[ 
F_P(B_\epsilon) = \begin{cases} P(i),   \text{   if   } i \in B_\epsilon \text{   with   } i \in I  \text{   or if   } B_\epsilon \subset (-\infty, i)  \text{   for   } i \in I  \\
P(k),  \text{   if   } B_\epsilon \subset (i,j) \text{   for   } i \neq j \in I \cup \{ \infty\} \text{ and } k = \min_I \{i,j\}\\
\end{cases} .
\]
Next, we describe the assignment of morphisms. If $B_\epsilon$ contains a zero-strata, $i$, or if $B_\epsilon'$ is entirely contained in some one-strata $(i,j)$, then $F_P(B_\epsilon \hookrightarrow B_\epsilon') = Id_{F(B_\epsilon)}$ . Suppose instead that $B_\epsilon'$ contains the vertex $i$ but  $B_\epsilon \subset (i,j)$. Then $F_P(B_\epsilon \hookrightarrow B_\epsilon') = (P(i) \to P(j))$ if $i < j$ (or $(P(j) \to P(i))$ if $j < i$). See Figure \ref{fig:cosheaf}. \label{cons:pmcosheaf}
\end{construction}

The cosheaf $F_P$ is locally constant on strata, so it defines a constructible cosheaf on the stratified space $(\RR, I)^{\wwedge}$.

\begin{example}
\begin{figure}[h!] \label{fig:cosheaf}
\centering
\includegraphics[width=.75\textwidth]{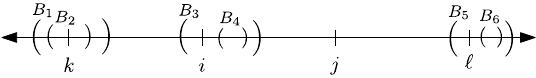}
\begin{caption}{Examples of open intervals occurring in Construction \ref{cons:pmcosheaf}.}
\end{caption}
\end{figure}
Suppose that $I$ is the poset $k > i > j < l$, where $k,i,j,$ and $l$ are ordered with the standard ordering on $\RR$ as in Figure \ref{fig:cosheaf}. Then $F_P(B_2 \hookrightarrow B_1) = Id_{F_P(B_2)} = Id_{P(k)}$, $F_P(B_4 \hookrightarrow B_3) = (P(j) \to P(i))$, and $F_P(B_6 \hookrightarrow B_5) = Id_{P(\ell)}$.
\end{example}

\subsection{Filtered Spaces and Cosheaves}
Next, we discuss how persistence-modules and persistence module cosheaves relate to filtrations of spaces.

\begin{definition}[Filtration]
Let $K$ be a simplicial complex. A \emph{filtration of $K$} is a sequence of subcomplexes $\{K_i\}_{i \in I}$ such that, for every $i$, there is an inclusion of spaces $K_i \hookrightarrow K_{i+1}$ and so that $K_{0} = \emptyset$ and $K_{\max \{i \in I\}} = K$.
\end{definition}

\begin{example}
If we take $I \subset \RR$ to be the indexing set of a filtration, then there is a natural way to view $I$ as a poset with the standard ordering of $\RR$. Passing to homology in degree $n$ defines an associated monotone persistence module via the assignment $i \mapsto H_n(K_i)$. The propagated persistence cosheaf on $\RR$ (see Construction \ref{cons:pmcosheaf}) is easy to describe.  Indeed, for a single one-stratum, we have $F(i,j) = H_n(K_i)$ and that the costalk of $F$ at a zero-stratum $i$ is $H_n(K_i)$. 
\end{example}

\subsubsection{Monotone and Index Filtrations}
Let $f: K \to \RR$ be a monotone function on simplices. That is, whenever $\tau$ is a face of $\sigma$, we have $f(\tau) \leq f(\sigma)$. Let $m_1 < m_2 < \ldots < m_p$ be the ordered set of minimum values in $\RR$ for which each $f^{-1}(-\infty, m_i]$ is a distinct non-empty simplicial complex. Setting $K_{m_0 = 0} = \emptyset$ and $K_{m_i} = f^{-1}(-\infty, m_i]$, we  define the \emph{(monotone) filtration of $K$ by $f$} as
$$ \emptyset = K_{m_0} \subset K_{m_1} \subset \ldots \subset K_{m_p} = K.$$
Notice that, by construction, all inclusion maps are in the direction of increasing index. Furthermore, if $K$ has $n$ non-empty simplices, $p \leq n+1$.

Next, suppose that $\emptyset=\sigma_{0} \prec \sigma_1 \prec \sigma_2 \prec \ldots \prec \sigma_n$ is a total order of the simplices of $K$ so that if either $f(\sigma_i) < f(\sigma_j)$, or if $\sigma_i$ is a face of $\sigma_j$, then $i<j$. Letting $K'_j = \{\sigma_i \mid i \leq j\}$, the increasing sequence of $n+1$ subcomplexes $\{K'_i\}$ is an \emph{index filtration compatible with the monotone filtration}.

In what follows, we will use $n$ to denote the number of non-empty simplices in a simplicial complex $K$ and $p$ to denote the number of steps in a monotone filtration.

\begin{remark}
Index filtrations are themselves monotone. Index filtrations are compatible with themselves, but to no other index filtrations.
\end{remark}

\subsubsection{From Index to Monotone}\label{ssec:cmap}
Suppose that $\{K_{m_j}\}_{m_j \in M}$ is a monotone filtration with filter function $f$ and $\{K'_i\}_{i \in [1,n]}$ is a compatible index filtration. Then for every $m_j \in M \setminus \{m_0\}$, there is some maximum interval $[\ell,r)$ for $\ell, r \in \{1, 2, \ldots, n\} \cup \{\pm \infty\}$ such that $f(\sigma_\ell) = f(\sigma_r) = m_j$ (where, whenever $r = \infty$ or $r = -\infty$, we define $\sigma_\infty := \sigma_n$ and $\sigma_{-\infty}= \sigma_0$, respectively). These intervals cover $\RR$, and every interval corresponds to a unique $m_j \in M$. Then we define a map of stratified spaces, $C: (\RR, \sf{Nat}([1,n])^{\wwedge}) \to (\RR, \sf{Nat}(M\setminus\{m_0\})^{\wwedge})$ that maps intervals with a particular value under the filter function $f$ to intervals with that same value under $f$. 

\begin{definition}\label{def:cmap}
Suppose that $a \in [\ell, r)$, where $[\ell, r)$ is the associated interval for some $m_j \in M$. Three cases arise: if $-\infty < \ell ,r < \infty$, we assign
\begin{align}
C(a) = \begin{cases} 
					m_j, \text{   if   } a < r-1\\
					m_j (r-a) + m_{j+1}(a - (r-1)), \text{   if   } a \geq r-1
				\end{cases}.
\end{align}
if $[\ell, r) = [-\infty, 1)$, we assign
\begin{align}
C(a) = am_1
\end{align}
and if $[\ell, r) = [n, \infty)$, we assign
\begin{align}
C(a) = \frac{a m_p}{n}.
\end{align}
\end{definition}
\begin{lemma}
$C$ is a stratified map. 
\end{lemma}

\begin{figure}[h!] 
\centering
\includegraphics[width=.9\textwidth]{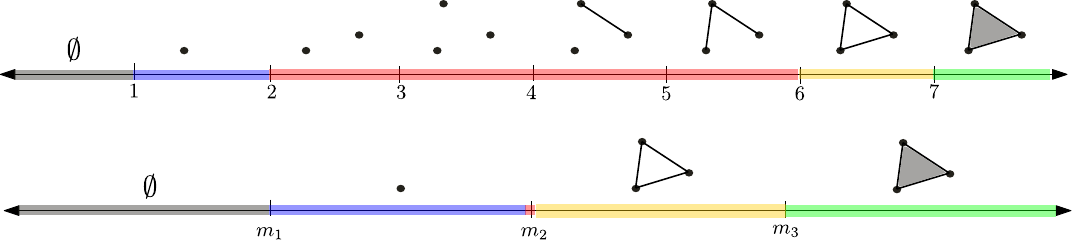}
\begin{caption}{An example of the map $C$. The relevant interval for the point, e.g., $m_2$ is~$[2,7)$, since the image of each simplex added in that interval under the filter function~$f$ is $m_2$. Then $[2,6)$ is mapped to $m_2$ and $[6,7)$ is mapped to $[m_2, m_3)$.\label{fig:cmap}}
\end{caption}
\end{figure}


\subsubsection{Augmented Descriptors via Index Filtrations}\label{sect:aug}

Let $\{K_{m_j}\}_{m_j \in M}$, $f$, and $\{K'_i\}_{i \in [0,n]}$ be a monotone filtration and compatible index filtration, respectively (as in the previous section).

Given a monotone filtration, we are perhaps interested in the so-called \emph{instantaneous events} that are captured in augmented topological descriptors, a remnant of the fact that many standard algorithms to produce descriptors for monotone filtrations are often actually employing compatible index filtrations. For example, an instantaneous $n$-dimensional homology event at time $m_j$ records the presence of an $n$-boundary that was not mapped from a boundary or cycle in the inclusion $K_{m_{j-1}} \hookrightarrow K_{m_{j}}$.

Note that many applications of TDA, such as the classic application of manifold learning through a Vietoris-Rips filtration, discard events with a short lifespan because they may be attributed to noise, so non-augmented persistence diagrams are the traditional tool of choice (see~\cite{chazal2013bootstrap}). However, recent developments in areas such as shape comparison and inverse TDA problems~(see, e.g.,~\cite{belton2020reconstructing}) rely on instantaneous events for efficient representation of simplicial or cubical complexes, particularly when the filtration used is directional (e.g., height filtration, lower-star filtration, etc.)

We aim to track both instantaneous and non-instantaneous events at every step of a monotone filtration. We introduce $A_n$ to account for instantaneous events (the extra data of an augmented module). Let $\beta_{m_j}$ denote the free group on $n$-dimensional boundaries of $K_{m_j}$ and let $\kappa_{m_{j-1}}$ denote the kernel of the map on $n$-dimensional homology induced by the inclusion $K_{m_{j-1}}\hookrightarrow K_{m_j}$. Since $\kappa_{m_{j-1}}$ corresponds to all cycles of $K_{m_{j-1}}$ that become boundaries in $K_{m_j}$, i.e., all cycles of $K_{m_{j-1}}$ that map to elements of $\beta_{m_j}$, the subgroup $\kappa_{m_{j-1}}$ can naturally be identified with a subgroup of $\beta_{m_j}$. Furthermore, since boundaries of $K_{m_{j-1}}$ are mapped injectively to boundaries of $K_{m_j}$, the subgroup $\beta_{m_{j-1}}$ is also naturally identified with a subgroup of $\beta_{m_j}$.
Then, we define:
\begin{equation}\label{eqn:monotoneaug}
A_n(K_{m_j}) =   \beta_{m_j}/  (\beta_{m_{j-1}} + \kappa_{m_{j-1}}).
\end{equation}
Note that, since $A_n(K_{m_j})$ is a quotient of free groups, and since the generators of $\beta_{m_{j-1}}$ and $\kappa_{m_{j-1}}$ are subsets of the generators of $\beta_{m_j}$, $A_n(K_{m_j})$ is free.
It may be helpful to think of $A_n(K_{m_j})$ as the free group on $n$-boundaries of $K_{m_j}$ that are not the images of boundaries or cycles in $K_{m_{j-1}}$. An instantaneous event in a monotone filtration is the appearence of an $n$-boundary that was not a boundary or cycle in the previous step of the filtration, meaning the rank of $A_n(K_{m_j})$ is the number of points (counting multiplicity) on the diagonal $(m_j, m_j)$ in the corresponding standard $n$-dimensional augmented persistence diagram.
We can also view $A_n$ as a repackaging of the ``entire'' information in index filtrations, independent of the choice of compatible index filtration. The connection to compatible index filtrations is made explicit in the following lemma.

\begin{lemma}\label{lem:coker}
Suppose that $\{K_{m_j}\}_{m_j \in M}$ is a monotone filtration corresponding to a filter function $f$ and $\{K'_i\}_{i \in [0,n]}$ is any compatible index filtration. Let $\kappa_i'$  denote the kernel of the map induced on homology in degree $n$ by the inclusion $K'_i \hookrightarrow K'_{i+1}$. Furthermore, let $\kappa'_{\hookrightarrow}$ denote the kernel of the map induced on homology in degree $n$ by the composition of inclusions $K'_{(\min C^{-1}(m_j))-1} \hookrightarrow \ldots \hookrightarrow K'_{\max C^{-1}(m_j)} $, where $C$ is as in Definition \ref{def:cmap} and Figure \ref{fig:cmap}. Then:
\begin{equation}\label{eqn:indexaug}
A_n(K_{m_j}) \cong \bigg( \bigoplus_{i = (\min C^{-1}(m_j))-1}^{\max C^{-1}(m_j)} \kappa_i' \bigg) / \kappa'_{\hookrightarrow}.
\end{equation}
\end{lemma} 

\begin{proof}
Recall that $A_n$ as defined in Equation \ref{eqn:monotoneaug} is a free group, so we first show the right side of Equation \ref{eqn:indexaug} is also a free group, and then show the desired isomorphism through a counting argument.

 We observe that generators of $\kappa'_{\hookrightarrow}$ correspond to cycles of $K'_{(\min C^{-1}(m_j))-1}$ that become boundaries somewhere along the composition of inclusions. Consider such a cycle and suppose that $K'_i$ is the last subcomplex in the filtraton where this cycle is still not a boundary -- then the cycle is naturally identified with a generator of $\kappa_i'$, since the cycle becomes a boundary in $K'_i \hookrightarrow K'_{i+1}$. This is true for each generator of $\kappa'_{\hookrightarrow}$, so we may view $\kappa'_{\hookrightarrow}$ as a subgroup of the sum in Equation \ref{eqn:indexaug}. Since the right side of Equation \ref{eqn:indexaug} is a quotient of free groups, where the generators of $\kappa'_{\hookrightarrow}$ are a subset of generators of the sum, the right side of the equation is a free group. We may therefore proceed by showing the left and right side of Equation \ref{eqn:indexaug} have equal rank.

Each step in an index filtration adds a single simplex, so either $\kappa_i' \cong \FF^0$ (if the simplex added does not fill in any $n$-cycle) or $\kappa_i' \cong \FF^1$ (if the simplex added in $K_i' \hookrightarrow K_{i+1}'$ fills in an $n$-cycle). Thus, the direct sum in the equation above has nontrivial terms only for values of $i$ such that $K'_i \hookrightarrow K'_{i+1}$ witnesses the death of $n$-cycles in the index filtration. Recall that $[\min C^{-1}(m_j), \max (C^{-1}(m_j))+1)$ is the maximum interval whose image under the filter $f$ is $m_j$. This means that, shifting to the left, we can identify $K'_{(\min C^{-1}(m_j))-1} = K_{m_{j-1}}$ and $K'_{\max C^{-1}(m_j)}= K_{m_{j}}$. Thus, every boundary of $K_{m_j}$ that was not present as a boundary in $K_{m_{j-1}}$ is introduced or becomes a boundary in some step of the index filtration between the values $(\min C^{-1}(m_j))-1$ and $\max C^{-1}(m_j)$, which means terms of the direct sum above are nontrivial only when boundaries are created. This is exactly the count of boundaries introduced in the inclusion $K_{m_{j-1}}\hookrightarrow K_{m_j}$, i.e., it is $\beta_{m_j}/\beta_{m_{j-1}}$, using the notation previously introduced in the paragraph above and Equation \ref{eqn:monotoneaug}. However, recall that $A_n({K_{m_j}})$ does not account for boundaries that fill in a cycle from a previous step in the filtration. Thus, we quotient out by the kernel of the composition of maps between $\min C^{-1}(m_j)$ and $\max (C^{-1}(m_j))+1$. This kernel is generated by boundaries and cycles of $K_{m_{j-1}}$ that are mapped to boundaries in $K_{m_{j}}$.
Since $K'_{(\min C^{-1}(m_j))-1} = K_{m_{j-1}}$ and $K'_{\max C^{-1}(m_j)}= K_{m_{j}}$, and since the index filtration is compatible with the monotone filtration, we see that $\kappa'_{\hookrightarrow} \cong \kappa_{m_{j-1}}$. Thus, the rank of the right side of Equation \ref{eqn:indexaug} is exactly the rank of the $A_n$ as defined in Equation \ref{eqn:monotoneaug}, and as both are free groups, we have shown the desired isomorphism.
\end{proof}

\begin{example}
Suppose that $\{K_{m_i}\}$ and $\{K'_i\}$ are monotone and index filtrations as in the bottom and top of Figure \ref{fig:cmap}. Then $A_0(K_{m_2}) \cong \beta_{m_2}/ (\beta_{m_1} +\kappa_{m_1}) \cong \FF^2 / \FF^0 \cong \FF^2$. Computed using the identification of Lemma \ref{lem:coker}, we see that this is the same as $\bigoplus_{i = 1}^6 \kappa_i' /\kappa'_{\hookrightarrow} \cong (\FF^0 \oplus \FF^1 \oplus \FF^1 / \FF_0 )\cong \FF^2$.
\end{example}

The following cosheaf organizes the information of both instantaneous and non-instantaneous events.

\begin{definition}[Augmented Filtration Cosheaf on $\RR$]
Let $\{K_{m_j}\}_{m_j \in M}$ be a monotone filtration of a simplicial complex $K$, and suppose $\RR$ is stratified by $\sf{Nat}(M \setminus \{m_0\})^{\wwedge}$. 
We define the \emph{augmented filtration cosheaf on $\RR$}, $F_A:  \sf{Opens}(\RR) \to \Vect$, on metric $\epsilon$-balls as follows.
$$F_A(U) = \begin{cases}
  H_n(K_{m_{j}}) \oplus A_n(K_{m_j}),  \text{   if   } m_{j-1} \in U  \\
  H_n(K_{m_j}) \text{   if   } U \subset (m_j, m_{j+1}) \text{   or   } U \subset (m_j=m_p, \infty) \\
  H_n(K_{m_1}) \text{   if   } U \subset  (-\infty, m_1)
\end{cases}.$$ 
\end{definition}

Observe that the above definition implies that the costalk at a zero-stratum $m_{j-1}$ of $(\RR, \sf{Nat}(M \setminus \{m_0\})^{\wwedge})$ is $H_n(K_{m_{j-1}}) \oplus A_n(K_{m_j})$. 

\begin{remark}
For an index filtration $\{K'_i\}_{i \in I}$, any new $n$-cycles introduced through the map $ K'_{i-1} \hookrightarrow K'_{i}$ are not $n$-boundaries, since the boundaries and interiors of simplices are added at distinct filtration events. Thus, $A_n(K'_i)$ is trivial, i.e., the augmented filtration cosheaf that arises from an index filtration is equivalent to its (non-augmented) filtration cosheaf. 
\end{remark}

An instance of the previous remark is illustrated by following example.

\begin{example}
Let $\{K'_i\}$ be the index filtration in the top of Figure \ref{fig:cmap}. Notice that the one-dimensional costalk of the non-augmented filtration cosheaf at $7$ is $H_1(K'_7) \cong \FF^0$, which is isomorphic to $H_1(K'_7) \oplus A_1(K'_7) \cong \FF^0 \oplus \FF^1/\FF^1 \cong \FF^0$.
\end{example}

The stratified map $C$ define above provides a clean interpolation between the augmented, non-augmented, and index cosheaves associated to a filtration. 

\begin{proposition}
Let $\cF_M$ and $\cF_A$ be the non-augmented and augmented filtration cosheaves for some monotone filtration $\{K_{m_j}\}_{m_j \in M}$ and let $\cF_I$ be the filtration cosheaf for a compatible index filtration $\{K_i\}_{i \in [0,n]}$. Let $C: (\RR, \sf{Nat}([1,n])^{\wwedge}) \to (\RR, \sf{Nat}(M \setminus \{m_0\})^{\wwedge})$ be the map of stratified spaces as above. Then,
\begin{itemize}
\item[(i)] We have an isomorphism of cosheaves $C_\ast \cF_I \cong \cF_M$;
\item[(ii)] Let $U \subset \RR$ be open such that $U \cap M = \emptyset$, then  $\cF_M (U) \cong \cF_A (U)$.
\end{itemize}
\end{proposition}

\begin{proof}
That $\cF_M$ and $\cF_A$ agree on one-strata follows directly from their definitions (they can differ at zero-strata). In claim (i), there are two parts: that $C_\ast \cF_I$ is constructible and that $C_\ast \cF_I$ is isomorphic to $\cF_M$. To prove the first, note that $C$ is a composition of ``elementary collapses" as described in Example \ref{ex:Cmap}, so by functoriality $C_\ast \cF_I$ is constructible. The second part of (i) is an explicit unwinding of the definition of the pushforward.
\end{proof}

\subsection{An Equivalence Result}

In this subsection we make explicit the relationship between zig-zag modules as put forth by Carlsson--Zamorodian and representations of the entrance path category of $\RR$ stratified by the natural numbers. In the process we will need to equip our zig-zag modules with additional structure, which we call ``markings."

Let $\mathsf{Poset}_I$ be the category of posets with Hasse diagrams homeomorphic to the interval, half-closed interval or $\RR$ and whose underlying set is at most countable. Morphisms in $\mathsf{Poset}_I$ are surjective maps of posets. So from above, the 
category of zig-zag modules is the category of pairs $(\cP , \rho)$ with $\cP \in \mathsf{Poset}_I$ and $\rho \colon \cP \to \Vect$ a representation of $\cP$.

\begin{definition}
Define the poset $\cZ \cZ_\NN$ to have objects $\frac{1}{2} \NN$ with non-identity morphisms
\[ \frac{a}{2} \le \frac{a+1}{2} \quad \text{ and } \quad \frac{a}{2} \le \frac{a-1}{2} \quad \text{ for all } a \in \NN, \text{ with } a \text{ odd}.
\]
\end{definition}

The poset $\cZ \cZ_\NN$ arises naturally when considering $\RR$ stratified (ambiently) by the natural numbers.

\begin{lemma}\label{lem:entrance}
There is a canonical isomorphism of categories
\[
\cZ \cZ_\NN \cong \mathsf{Ent}_\Delta (\RR, \Nat(\NN)^{\wwedge}).
\]
\end{lemma}

We wish to ``mark" our posets by passing to the {\em under category} of $\cZ \cZ_\NN$, $\mathsf{Poset}^{\cZ\cZ_\NN /}_I$, i.e., we will consider posets equipped a map from $\cZ \cZ_\NN$. Passing to marked objects/the under category has the effect of replacing a given poset by all possible labelings of that poset by the natural numbers. (The notion of marking persistence modules is not at all unusual. For instance, in most applications, the passage from persistence modules to barcodes or diagrams depends on an explicit marking, e.g., the event times/parameters.) Morphisms in the under category are commutative triangles. As we will use later, passage to the under category introduces an initial object: $\mathrm{Id} : \cZ \cZ_\NN \to \cZ \cZ_\NN$.
Note that the under category is an example of a comma category and are also known as {\em coslice} categories. 

\begin{definition}
Define the category of \emph{marked zig-zag modules}, $\mathsf{ZZmod}$, to be the category of pairs $(\cZ\cZ_\NN \twoheadrightarrow \cP , \rho)$ with $\cZ\cZ_\NN \twoheadrightarrow \cP \in \mathsf{Poset}^{\cZ\cZ_\NN /}_I$ and $\rho \colon \cP \to \Vect$ a representation. A morphism is a pair $(f,\varphi) \colon (\cZ\cZ_\NN \twoheadrightarrow \cP , \rho) \to (\cZ\cZ_\NN \twoheadrightarrow \cQ , \eta)$ with $f \colon \cP \twoheadrightarrow \cQ$ defining a morphism in the under category and $\varphi \colon \rho \Rightarrow f^\ast \eta$ a natural transformation.
\end{definition}

It turns out that isomorphism in $\mathsf{ZZmod}$ is too strong to capture our preferred notion of ``sameness," so we introduce a notion of weak equivalence. An example of an operation that creates a weakly equivalent module is ``subdividing" a vertex in a poset into several vertices provided that all of the new maps in the corresponding representation are isomorphisms.

\begin{definition}
A morphism $(f,\varphi) \colon (\cZ\cZ_\NN \twoheadrightarrow \cP , \rho) \to (\cZ\cZ_\NN \twoheadrightarrow \cQ , \eta)$ in $\mathsf{ZZmod}$ is a \emph{weak equivalence} if $\varphi \colon \rho \Rightarrow f^\ast \eta$ is a natural isomorphism. Let $\cW$ denote the collection of weak equivalences.
\end{definition}

We caution the data-analytically oriented reader here; notice that weakly equivalent objects of $\mathsf{ZZmod}$ do not generally have the same indices of ``events,'' i.e., vertices at which the corresponding image of the representation changes. That is, the standard map from $\mathsf{ZZmod}$ to persistence diagrams (as described in \cite{carlsson2010zigzag}) does not factor through $\mathsf{ZZmod} [\cW^{-1}] $. However, the order and number of events is preserved.

%
%
%
%
%

\begin{theorem}\label{thm:equivcat}
The category of (marked) zig-zag modules localized at weak equivalences is equivalent to the category of constructible cosheaves on $\RR$ stratified by the natural numbers. That is, we have an equivalence of categories
\[
\mathsf{ZZmod} [\cW^{-1}] \cong \mathsf{Fun} (\mathsf{Ent}_\Delta (\RR, \Nat(\NN)^{\wwedge}),\Vect) \simeq \mathsf{cShv}_{cbl}^{\Vect} ((\RR, \NN)^{\wwedge}).
\]
\end{theorem}

The second equivalence is just an example of the exodromy equivalence. The first equivalence, which is actually an isomorphism of categories, is proved in Appendix \ref{app:B}. There are some technicalities in proving the previous theorem, but the main idea of the equivalence  is as follows. Let $\varphi \colon \cP \to \Vect$ be a representation of $\cP$. Pullback~$\varphi$ along the map $\cZ \cZ_\NN \to \cP$ to obtain a representation of $\cZ \cZ_\NN$ so that, via Lemma \ref{lem:entrance}, we have a representation of the corresponding entrance path category, i.e., a constructible cosheaf.

\section{K-theory of Zig-Zag Modules}\label{sect:ktheory}

We now shift gears and compute the K-theory of the category of zig-zag modules. The category in which our modules take values plays a central role and we consider two different constructions:  one for modules valued in vector spaces and one for set-valued modules . To begin, we work with an arbitrary combinatorial manifold as parameter space and only later specialize to the case where it is one-dimensional. When our parameter space is one-dimensional, it's combinatorial entrance path category is a zig-zag poset and hence a representation is a zig-zag module.

Motivated by the Exodromy Theorem and our equivalence result above, we make the following definition.

\begin{definition}\label{def:kofzz}
Let $(X \xrightarrow{\phi} \cP)$ be a combinatorial manifold with its native stratification, $\mathsf{Ent}_\Delta (X, \cP)$ its combinatorial entrance path category, and   $V$ any category. The {\em category of $V$-valued persistence modules parameterized by $X$}, $\mathsf{pMod}^{V}(X)$, is given by
\[
\mathsf{pMod}^{V}(X):=\mathsf{Fun} (\mathsf{Ent}_\Delta (X, \cP),V).
\]
Hence, the \emph{$K$-theory of $V$ valued persistence modules (parametrized by $X$)} is the $K$-theory spectrum (whenever it exists) of the category above: $\KK(\mathsf{pMod}^{V}(X))$.
\end{definition}


\subsection{$K$-Theory of $\Vect$-Valued coSheaves}

The category of finitely generated modules for a commutative ring is an Abelian category, so we  define/compute K-theory using the work of Quillen and Waldhausen. (If our ring is a field, we recover our old friend $\Vect$).    In this section we will freely use the material of Appendix \ref{app:A}.

\begin{lemma}\label{lem:key}
Let $R$ be a commutative ring, $\cM$ the associated Waldhausen category of finitely generated modules, $X$ a combinatorial manifold, and $x_0 \in X$ a connected zero-dimensional stratum, i.e., a point that is a stratum. The following sequence is split short exact sequence of Waldhausen categories
\[
\xymatrix{\mathsf{Fun} (\mathsf{Ent}_\Delta (X \setminus x_0),\cM)\ar[r]_{\quad j_\ast} 
&
\mathsf{Fun} (\mathsf{Ent}_\Delta (X),\cM) \ar[r]_{i^\ast} \ar@/_1.0pc/[l]_{\quad j^\ast} &
 \mathsf{Fun} (\mathsf{Ent}_\Delta (x_0),\cM) 
\ar@/_1.0pc/[l]_{i_\ast},
}
\]
where $i \colon x_0 \hookrightarrow X$ and $j \colon X \setminus x_0 \hookrightarrow X$ are the inclusion maps.
\end{lemma}

\begin{proof}
The content of Lemma \ref{lem:toWald} is precisely that the three categories appearing are Waldhausen. We next observe that the inverse and direct image functors (in this setting) are compatible with the equivalences and cofibrations, so indeed we have a sequence of exact functors.

It is standard that $i_\ast$ is right adjoint to $i^\ast$ and in this case, the counit of the adjunction is a natural isomorphism. Because our domain categories are discrete (finite even), $j_\ast$ is indeed left adjoint to $j^\ast$ and the unit is a natural isomorphism; $j_\ast$ is the extension by zero map. The composition $i^\ast \circ j_\ast$ is manifestly the zero functor and $i^\ast$ presents $ \mathsf{Fun} (\mathsf{Ent}_\Delta (x_0),\cM)$ as the cokernel of $j_\ast$. In summary, the sequence is  short exact and the adjointness properties we observed further show it is split.

\end{proof}

\begin{remark}
The preceding lemma is straightforward as we are considering constructible cosheaves  on the complement of a point (which is open). One could try to prove a version of the lemma above where $x_0$ is replaced by an arbitrary stratum and at the level of non-combinatorial entrance path categories, but---in general---this fails as $j^\ast$ will not have the appropriate adjointness properties. There is, however, a corresponding lemma for an arbitrary closed/open complement decomposition that is compatible with the stratification.
\end{remark}

\begin{lemma}\label{lem:standard}
The split short exact sequence of Lemma \ref{lem:key}  is standard.
\end{lemma}

\begin{proof}
Condition (3) of Definition \ref{defn:standard} holds for categories of modules (see Remark 2.18 of \cite{fiore}) and by the same reasoning, our category of functors valued in $\cM$.

Let $\cF \in \mathsf{Fun} (\mathsf{Ent}_\Delta (X),\cM)$. Each component of the natural transformation $\left (j_\ast \circ j^\ast \right )(\cF) \to \cF$ is an isomorphism, except for the component corresponding to $x_0$. The component corresponding to $x_0$ is the inclusion of zero, which is a cofibration. Therefore, $\left(j_\ast \circ j^\ast \right )(\cF) \to \cF$ is a cofibration in the functor category.

Finally, let $\varphi \colon \cF \to \cF'$ be a cofibration in $\mathsf{Fun} (\mathsf{Ent}_\Delta (X),\cM)$. We need to check that unique map 
\[
\psi \colon \cF \coprod_{j_\ast j^\ast \cF} j_\ast j^\ast \cF' \to \cF'
\]
is a cofibration. By definition, we must check this condition componentwise. For a component corresponding to $x_0 \neq S \subset X$, 
the kernel of $\psi$ is exactly the submodule of $\cF (S) \oplus \cF' (S)$ by which we quotient when constructing pushouts in categories of modules; that is, the $S$ component of $\psi$ is a monomorphism.  For the $x_0$ component, the pushout is identified with $\cF(x_0)$ and $\psi_{x_0} = \varphi_{x_0}$, so by hypothesis it is a monomorphism.
\end{proof}

We require one final observation/lemma before assembling the proof of Theorem \ref{thm:bigOplus}. From the definition of entrance paths and the fact that K-theory preserves colimits, it immediately follows that K-theory is addivitive with respect to connected components of our parameter space. That is:

\begin{lemma}\label{lem:additive}
Let $X = X_a \amalg X_b$ be a  stratified space, then there is an equivalence of spectra
\[
\KK(\mathsf{pMod^{\cM}} (X)) \cong \KK(\mathsf{pMod^{\cM}} (X_a)) \vee \KK(\mathsf{pMod^{\cM}} (X_b)).
\]
\end{lemma}

Although the strata of a one-dimensional stratified space are not generally disjoint, we still have an additivity result similar to the previous lemma, as we will now show.

\begin{lemma}\label{lem:stratadd}
Let $X$ be a one-dimensional combinatorial manifold. There is an  equivalence of spectra
\[
\KK(\mathsf{pMod^{\cM}} (X)) \cong \bigvee_{x_0 \in X_0} \KK(\mathsf{pMod^{\cM}} (x_0)) \vee \bigvee_{x_1 \in X_1} \KK(\mathsf{pMod^{\cM}} (x_1)).
\]
where $X_i$ is the set of $i$-strata of $X$.
\end{lemma}

\begin{proof}
We proceed by induction over the number of zero-strata. As our base case, note that when there are no zero-strata, we have $X_0 = \emptyset$ and $x_1 = X_1 = X$, so the claim holds.
Now suppose that the claim holds whenever $X$ contains $n-1$ zero-strata, for all $n-1 \geq 0$. Then consider the case that $X$ contains $n$ zero-strata. For an arbitrary zero-stratum $x_0^* \in X_0$, we know by Lemma \ref{lem:key} that 
\[
\xymatrix{\mathsf{Fun} (\mathsf{Ent}_\Delta (X \setminus x_0^*),\cM)\ar[r]_{} 
&
\mathsf{Fun} (\mathsf{Ent}_\Delta (X),\cM) \ar[r]_{} \ar@/_1.0pc/[l]_{} &
 \mathsf{Fun} (\mathsf{Ent}_\Delta (x_0^*),\cM) 
\ar@/_1.0pc/[l]_{}
}
\]
is a split short exact sequence of Waldhausen categories. Then by Theorem \ref{thm:additivity}, we see that there is an equivalence of spectra
\[
\KK(\mathsf{pMod^{\cM}} (X)) \cong \KK(\mathsf{pMod^{\cM}} (X \setminus x_0^*)) \vee \KK(\mathsf{pMod^{\cM}} (x_0^*))
\]
Since $X \setminus x_0^*$ is itself a one-dimensional combinatorial manifold with $n-1$ zero-strata, our inductive hypothesis allows us to write
\begin{align*}
\KK(\mathsf{pMod^{\cM}} (X)) \cong  \bigg(\bigvee_{x_0 \neq x_0^* \in X_0}\KK(\mathsf{pMod^{\cM}} (x_0^*)) \vee \KK(\mathsf{pMod^{\cM}} &(X_1))\bigg) 
\\ 
&\vee \KK(\mathsf{pMod^{\cM}} (x_0^*)).
\end{align*}
Since the zero-strata are disjoint, by Lemma \ref{lem:additive}, we may reindex by absorbing the last term into the first and we have the desired result.
\end{proof}

Utilizing the preceeding lemma, we now prove the following theorem which computes the $K$-theory of zig-zag modules parametrized by a given 1-manifold.

\begin{theorem}\label{thm:bigOplus}
Let $X$ be a one-dimensional combinatorial manifold. There is an  equivalence of spectra
\[
\KK(\mathsf{pMod^{\Vect_\FF}} (X)) \cong \bigvee_{X_0} \KK (\FF) \vee \bigvee_{X_1} \KK (\FF),
\]
where $X_i$ is the set of $i$-strata of $X$ and where $\KK (\FF)$ denotes the K-theory spectrum of the field $\FF$.
\end{theorem}

\begin{proof}
First, we identify the $K$-theory of components of the stratification, i.e., we identify $\KK(\mathsf{pMod^{\Vect}} (x_0))$ and $\KK(\mathsf{pMod^{\Vect}} (x_1))$ for $x_0 \in X_0$ and $x_1 \in X_1$, respectively. We begin with the former.

By Definition \ref{def:kofzz}, we have $\KK(\mathsf{pMod^{\Vect_\FF}} (x_0)) = \KK(\mathsf{Fun} (\mathsf{Ent}_\Delta (x_0),\Vect_\FF))$. Since $\mathsf{Ent}_\Delta(x_0)$ is the terminal category (a single object and an identity morphism), $\mathsf{Fun} (\mathsf{Ent}_\Delta (x_0),\Vect_\FF)$ is isomorphic to the category of $\Vect_\FF$ itself. Thus, $\KK(\mathsf{pMod^{\Vect}_\FF} (x_0)) = \KK(\Vect)$. Now, the category of finite dimensional vector spaces over $\FF$ is exactly the category of finitely generated projective modules over $\FF$ (considered as a ring). Hence, $\KK(\Vect_\FF)$ is just the algebraic K-theory of $\FF$.

We observe that $\mathsf{Ent}_\Delta(x_1)$ is also a single object category, so the proof that $\KK(\mathsf{pMod^{\Vect}_\FF} (x_1)) \cong \KK(\FF)$ is  identical.
Thus, we have shown the $K$-theory of each strata is a copy of $K(\FF)$. We know by Lemma \ref{lem:stratadd} that $\KK(\mathsf{pMod^{\Vect_\FF}} (X))$ is additive over strata, so the result follows.
\end{proof}

\begin{remark}
An alternative approach to proving the preceding theorem  could be to use Serre subcategories and Abelian Localization. This approach has a number of its own subtleties so we have presented the proof above.
\end{remark}

\begin{corollary}
For $X$ a one-dimensional combinatorial manifold, we have
\[
K_0(\mathsf{pMod^{\Vect_\FF}} (X)) \cong \bigoplus_{X_0} \ZZ \oplus \bigoplus_{X_1} \ZZ,
\]
and
\[
K_1(\mathsf{pMod^{\Vect_\FF}} (X)) \cong \bigoplus_{X_0} \FF^\times \oplus \bigoplus_{X_1} \FF^\times,
\]
where $X_i$ is the set of $i$-strata of $X$ and $\FF^\times$ is the group of units of $\FF$.
\end{corollary}

The higher K-theory of fields contains interesting torsion and other phenomena. We refer the reader to Chapter IV of \cite{Weibel} for an in-depth description.

\subsection{Pointed Set Valued coSheaves}

While persistence modules are most often assumed to take values in vector spaces, there are interesting modules/cosheaves that take values in other categories. Of particular interest to us is the {\em Leray--Reeb} cosheaf, $\cL_f$, associated to a map $f \colon Y \to X$, see \cite{CP}. Let us consider a simple situation: let $f \colon Y \to \RR$ be a Morse function on a closed manifold $Y$.  Now, given $U \subseteq \RR$, let $\cL_f(U) := \pi_0 f^{-1} (U)$. It is standard that the critical values of $f$ stratify $\RR$ and that $\cL_f$ is constructible with respect to this stratification. So the Leray--Reeb cosheaf defines a persistence module taking values in the category of finite sets $\mathsf{Set}$.

For technical convenience we prefer our sets to be pointed/based. Let us  consider $\mathsf{Set}_\ast$, the category of finite pointed sets and base point preserving functions. The category $\mathsf{Set}_\ast$ is Waldhausen (cofibrations are injections and weak equivalences are bijections), so given a combinatorial manifold $(X \xrightarrow{\phi} \cP)$ we can compute the $K$-theory of the associated (Waldhausen) category of persistence modules $\mathsf{pMod}^{\mathsf{Set}_\ast} (X)$.

Note that the proof Lemma \ref{lem:key} goes through for $\mathsf{Set}_\ast$ valued functors {\em mutatis mutandis}. Similarly, Lemma \ref{lem:additive} is easily adapted to the case at hand. The following version of Lemma \ref{lem:standard} requires only slightly more care.

\begin{lemma}\label{lem:setadd}
Let $X$ be a combinatorial manifold and $x_0 \in X$ a connected zero-dimensional stratum, i.e., a point that is a stratum. The following   split short exact sequence of Waldhausen categories is standard
\[
\xymatrix{\mathsf{Fun} (\mathsf{Ent}_\Delta (X \setminus x_0),\mathsf{Set}_\ast)\ar[r]_{\quad j_\ast} 
&
\mathsf{Fun} (\mathsf{Ent}_\Delta (X),\mathsf{Set}_\ast) \ar[r]_{i^\ast} \ar@/_1.0pc/[l]_{\quad j^\ast} &
 \mathsf{Fun} (\mathsf{Ent}_\Delta (x_0),\mathsf{Set}_\ast) 
\ar@/_1.0pc/[l]_{i_\ast},
}
\]
where $i \colon x_0 \hookrightarrow X$ and $j \colon X \setminus x_0 \hookrightarrow X$ are the inclusion maps.
\end{lemma}

\begin{proof}
Condition (3) of Definition \ref{defn:standard} is inherited from $\mathsf{Set}_\ast$ where a cofibration is an injection and a cofiber sequence of finite pointed sets $S \hookrightarrow T \to \ast$ requires a bijection $S \cong T$.

Let $\cF \in \mathsf{Fun} (\mathsf{Ent}_\Delta (X),\mathsf{Set}_\ast)$. As before, each component of the natural transformation $\left (j_\ast \circ j^\ast \right )(\cF) \to \cF$ is an isomorphism, except for the component corresponding to $x_0$. The component corresponding to $x_0$ is the inclusion of zero (the singleton set $\ast$), which is a cofibration. Therefore, $\left (j_\ast \circ j^\ast \right )(\cF) \to \cF$ is a cofibration in the functor category.

Finally, let $\varphi \colon \cF \to \cF'$ be a cofibration in $\mathsf{Fun} (\mathsf{Ent}_\Delta (X),\mathsf{Set}_\ast)$. We need to check that unique map 
\[
\psi \colon \cF \coprod_{j_\ast j^\ast \cF} j_\ast j^\ast \cF' \to \cF'
\]
is a cofibration. The same (componentwise) argument works as before. That is, for the stratum $x_0$  the pushout is identified with $\cF(x_0)$ and $\psi_{x_0}= \varphi_{x_0}$. Let $x_0 \neq S \in \mathsf{Ent}_\Delta(X)$, we are left to consider the commutative diagram below, where the square is a pushout,
\[
\xymatrix{ \cF(S) \ar[r]^{\varphi_S} \ar[d]_{\mathrm{Id}} & \cF' (S) \ar[d] \ar@/^1.0pc/[ddr]^{\mathrm{Id}}&\\
\cF(S) \ar[r] \ar@/_1.0pc/[drr]^{\varphi_S}& \cF(S) \coprod_\varphi \cF'(S) \ar[dr]^{\psi_S}&\\
&&\cF'(S).}
\]
Hence, as $\varphi_S$ is injective, so is $\psi_S$.

\end{proof}

Arguing as in the preceding subsection, we deduce the following.

\begin{lemma}\label{lem:strataddset}
Let $X$ be a one-dimensional combinatorial manifold. There is an  equivalence of spectra
\[
\KK(\mathsf{pMod^{\mathsf{Set}_\ast}} (X)) \cong \bigvee_{x_0 \in X_0} \KK(\mathsf{pMod^{\mathsf{Set}_\ast}} (x_0)) \vee \bigvee_{x_1 \in X_1} \KK(\mathsf{pMod^{\mathsf{Set}_\ast}} (x_1)).
\]
where $X_i$ is the set of $i$-strata of $X$.
\end{lemma}

The Barratt--Priddy--Quillen--Segal Theorem (see Chapter 4 of \cite{Weibel}) proves that there is an equivalence of spectra
\[
\SS \cong \KK(\mathsf{Set}_\ast) \cong \KK(\mathsf{pMod^{Set_\ast}} (x_0)),
\]
for $x_0 \in X$ a connected zero stratum in a combinatorial manifold $X$, and where $\SS$ is the sphere spectrum. Recall that the homotopy groups of $\SS$ are the stable homotopy groups of spheres. Consequently, by assembling our work to this point, we have proven the following.

\begin{theorem}\label{thm:set}
Let $X$ be a one-dimensional combinatorial manifold. There is an  equivalence of spectra
\[
\KK(\mathsf{pMod^{Set_\ast}} (X)) \cong \bigvee_{X_0} \SS \vee \bigvee_{X_1} \SS,
\]
where $X_i$ is the set of $i$-strata of $X$ and where $\SS$ denotes the sphere spectrum. In particular, 
\[
K_0(\mathsf{pMod^{Set_\ast}} (X)) \cong \bigoplus_{X_0} \ZZ \oplus \bigoplus_{X_1} \ZZ,
\]
and
\[
K_1(\mathsf{pMod^{Set_\ast}} (X)) \cong \bigoplus_{X_0} \ZZ/2 \oplus \bigoplus_{X_1} \ZZ/2.
\]
\end{theorem}

As it is the central object in homotopy theory, much is known about $\SS$, though mysteries remain. A remarkable theorem of Serre implies that $\pi_n (\SS)$ is finite for $n>0$ and these groups are known up to around $n=100$.

\begin{remark}
If one wants to avoid using pointed sets/basepoints, one can consider the plain old category of sets $\mathsf{Set}$ and functions. This category does not have a zero object as the initial object is the empty set, while a final object is a singleton set. Hence, $\mathsf{Set}$ does not define a Waldhausen category in a straightforward manner. If one considers the subcategory $\mathsf{Set}_i$ consisting of the same objects, but where a morphism must be injective, one can define $\KK (\mathsf{pMod}^{\mathsf{Set}_i} (X))$. Indeed, $\mathsf{Set}_i$ and the resulting functor category can be equipped with the structure of an {\em assembler} and Zakharevich defines $K$-theory for assemblers in \cite{Inna1} and \cite{Inna2}. It is again a consequence of the Barratt--Priddy--Quillen--Segal Theorem that for each $n$ we have an isomorphism
\[
K_n (\mathsf{pMod}^{\mathsf{Set}_i} (X)) \cong K_n (\mathsf{pMod}^{\mathsf{Set}_\ast} (X)).
\]

\end{remark}

\section{Euler Curves and Virtual Diagrams}

In this section, we give two applications of our K-theoretic work.

\subsection{Euler Curves and $K_0$}

Let $\{V_i\}$ be a monotone persistence module of vector spaces, with indexing set $I$. We choose an embedding $I \hookrightarrow \NN$ and---in what follows---identify $I$ with its image in the natural numbers. The propagated persistence cosheaf, $F_V$, is constructible on $\RR$ stratified by $\NN$.

\begin{definition}
The (scaled) \emph{Euler curve} of $\{V_i\}$ is the constructible function $\chi_V \colon \RR \to \ZZ$ given by $\chi (x) = \mathop{rank}  (F_V)_x$, the rank of the costalk at $x \in \RR$. If $\{V_i\}$ is a module of graded vector spaces, then $\chi_V$ is the alternating sum of the ranks of the graded vector space that is the costalk.
\end{definition}

Note that any constructible $\ZZ$-valued function on $\RR$ naturally defines a class in $K_0$. As noted in the proof of Theorem \ref{thm:bigOplus}, the class in $K_0$ of a cosheaf only depends on its dimension, so we have the following.

\begin{proposition}
Let $V_\ast$ be a standard, finite persistence module of vector spaces and $\RR$ stratified ambiently by its subset $\NN$. Then,
\[
[F_V]=[\chi_V] \in K_0 (\mathsf{pMod}^{\Vect}(\RR)).
\]
\end{proposition}

While the statement of the proposition feels obvious, it does contain content. Indeed, one of the classical motivations for  simplicial homology is fixing the functoriality of the Euler characteristic. In general, a map of complexes $f \colon X_\bullet \to Y_\bullet$ does not induce a map between $\chi (X_\bullet)$ and $\chi (Y_\bullet)$; only if $f$ is covering map is there a multiplicative relationship between Euler characteristics. The {\it categorification} of the Euler characteristic to homology fixes this functoriality issue. Given any (co)homological setting there is an analogue of Euler class (in topology, this can be achieved by considering orientations for cohomology theories). The preceding proposition witnesses a K-theoretic Euler class.

One consequence of realizing Euler curves/classes K-theoretically is that there is an obvious extension to arbitrary (finite) persistence modules: zig-zag, higher dimensional, etc. (As before, we only see the scaled/standardized curve/class.) This construction is an explicit realization of the yoga that K-theory is the universal Euler characteristic. 


\begin{definition}\label{defn:euler}
Let $X$ be a combinatorial manifold, $V$ a category,  and $\cF \in \mathsf{pMod}^{V}(X)$  a persistence module. The \emph{Euler class},~$\chi (\cF)$, of $\cF$ is the $K-class$
\[
\chi(\cF) := [\cF] \in K_0 (\mathsf{pMod}^{V}(X)).
\]
\end{definition}

\subsection{Virtual Diagrams}\label{sect:virtual}

In \cite{BubVirtual}, Bubenik and Elchesen describe the group completion of a monoid of persistence diagrams. The resulting equivalence classes are called {\it virtual persistence diagrams} and can be realized by extending the diagrams to include arbitrary points in the (extended) first quadrant, i.e., not just points above the diagonal. We will denote Bubenik and Elchesen's Abelian group of virtual persistence diagrams by $K_0 (\mathsf{Diag})$.  We now describe a homomorphism (and its image)
\[
\delta \colon K_0 (\mathsf{pMod}^{\Vect}_{fin}(\RR) \to K_0 (\mathsf{Diag}),
\]
where $\RR$ is stratified by its subset $\NN$, i.e., the parameter space is $(\RR, \NN)^{\wwedge}$.

To begin, let $\mathsf{pMod}^{\Vect}_{fin}(\RR)$ denote the category of $\Vect$-valued constructible cosheaves on our stratified $\RR$ that are eventually constant, i.e., there exists $N \in \NN$ such that beyond $N$ the cosheaf is constant. This category has a monoidal structure induced by $\oplus$ in $\Vect$, so the objects in the category form a (commutative, unital) monoid.

We require a small tweak to the category $\mathsf{Diag}$ from \cite{BubVirtual}. As we allow features to persist for all future time, our persistence diagrams are built from the extended real line $\RR \cup \{\infty\}$; this is a minor point and we suppress it from notation.

Now, as noted above, we have an identification of $\mathsf{Ent}_\Delta (\RR, \sf{Nat}(\NN)^{\wwedge})$ with the poset~$\cZ\cZ_\NN$. Hence, an object $\cF \in \mathsf{pMod}^{\Vect}_{fin}(\RR))$ is simply a representation of~$\cZ\cZ_\NN$ (which is eventually finite). Following \cite{carlsson2010zigzag}, we use indecomposables of the associated representation of $\cZ\cZ_\NN$ to associate a diagram to $\cF$. More explicitly, we have the following assignment of a multi-set of points to a cosheaf
\[
\check{\delta} \colon \mathsf{pMod}^{\Vect}_{fin}(\RR) \to \sf{Diag} \subset K_0 (\sf{Diag}), \quad \cF \mapsto \{(b_i, d_i)\},
\]
 where each $b_i$ and $d_i$ correspond to the left and right indices (respectively) of an indecomposable element of the associated representation of $\cZ\cZ_\NN$.

\begin{figure}[h!] 
\centering
\includegraphics[scale=.75]{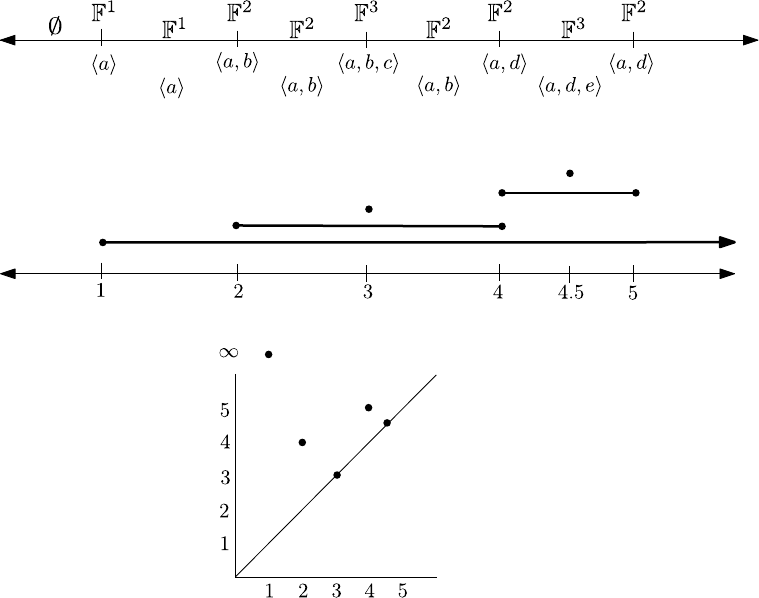}
\begin{caption}{The result of applying $\delta$ to the cosheaf shown on the top of the figure is the persistence diagram shown on the bottom. In the middle, we have drawn the associated barcode. In the spirit of \cite{carlsson2010zigzag}, we have shown all bars as closed intervals to emphasize that they do not necessarially arise from a monotone filtration. Note the presence of length-zero barcodes and on-diagonal points, corresponding to indecomposable elements with a single vector space. \label{fig:cosheaftodiag}}
\end{caption}
\end{figure}
Note that $\check{\delta}$ can easily be adapted to be a map into barcodes, where, instead of a point $(b, d)$, we draw a bar between $b$ and $d$.
This map $\check{\delta}$ (and the adaptation to signed barcodes) is nearly identical to the one described in Definition 2.6 of \cite{carlsson2010zigzag} with two notable differences. Firstly, the diagrams of ibid have points only on the integer lattice, whereas our diagrams have points on the $1/2$-integer lattice. This is a consequence of us additionally considering edges of the stratification rather than only vertices, and of our convention to then index vertices by non-integers. Furthermore, the diagrams of ibid contain on-diagonal points only when the maps to a particular vertex both have a nontrivial kernel. Our diagrams allow for these type of on-diagonal points, but additionally allow for on-diagonal points when the maps from an edge to its endpoints both have a nontrivial kernel. 

In general, these differences may be attributed to begining with persistence modules (the starting place for the map in \cite{carlsson2010zigzag}) or begining with persistence module cosheaves (the starting place for our map $\check{\delta}$). When one begins with persistence modules, the resulting cosheaf is a specific type -- importantly, each edge has an identity morphism to at least one of its endpoints (see Construction \ref{cons:pmcosheaf}). This allows ibid to only consider a poset on vertices, which may be otained from our poset of vertices and edges by collapsing these identity morphisms. The language of ibid is therefore more compatible with an explicit connection to filtrations, whereas our setting is generalized.

Returning to our primary goal, we note that our diagram map $\delta$ takes direct sums to sums of multisets.

\begin{lemma}
The map 
\[
\check{\delta} \colon \mathsf{pMod}^{\Vect}_{fin}(\RR) \to K_0 (\sf{Diag})
\]
is a monoid homomorphism.
\end{lemma}

By the universal property of $K_0$ we obtain our desired homomorphism.

\begin{corollary}
There is a homomorphism of Abelian groups
\[
\delta \colon K_0(\mathsf{pMod}^{\Vect}_{fin}(\RR))  \to K_0 (\sf{Diag})
\]
such that the following diagram commutes
\[
\xymatrix{
\mathsf{pMod}^{\Vect}_{fin}(\RR) \ar[r]^{\quad \check{\delta}} \ar[d] & K_0(\sf{Diag})\\
K_0(\mathsf{pMod}^{\Vect}_{fin}(\RR)) \ar[ur]_{\delta}
}.
\]
\end{corollary}

Because we have scaled/standardized our modules---as reflected by the parameter space $(\RR, \NN)^{\wwedge}$---the map $\delta$ has zero chance of being surjective (let alone an isomorphism). Following \cite{BubVirtual}, let $\mathsf{Diag}\left( \left (\frac{1}{2}\ZZ \right )^2, \left (\frac{1}{2}\ZZ_{\ge 0} \right)^2 \right )$ be monoid of (classical) persistence diagrams with half-integer (or infinite) coefficients. This monoid is a submonoid of the monoid of all (classical) persistence diagrams, $\mathsf{Diag}$, and we let $\mathsf{G} < K_0(\mathsf{Diag})$ denote the subgroup it generates. The following is clear from construction.

\begin{proposition}
The subgroup $\mathsf{G} < K_0 (\mathsf{Diag})$ is isomorphic to the image of the homomorphism
\[
\delta \colon  K_0(\mathsf{pMod}^{\Vect}_{fin}(\RR)) \to K_0 (\mathsf{Diag}).
\]
\end{proposition}

\appendix

\section{Waldhausen's K-theory}\label{app:A}

In this appendix we outline Waldhausen's construction of algebraic K-theory. In particular, we build to a fundamental additivity result: Waldhausen Additivity. The work of Waldhausen was first published in 1985 \cite{Wald}. Our notation follows the much more recent article of Fiore and Pieper \cite{fiore}.

\begin{definition} A Waldhausen category, $\sC$, is a category equipped with a subcategory of weak equivalences, $w(\sC)$, a subcategory of cofibrations, $co(\sC)$, and a distinguished zero object. Further, the triple $(\sC, co(\sC), w(\sC))$ must satisfy
\begin{enumerate}
\item Every isomorphism in $\sC$ is a cofibration;
\item Each object $c \in \sC$ is cofibrant, i.e., the unique map $0 \to c$ is a cofibration;
\item Cokernels exist and define cofibration sequences; and
\item Weak equivalences glue along cofibrations.
\end{enumerate}
\end{definition}

Waldhausen categories are a more general setting for algebraic K-theory than Abelian and exact categories. In particular, if $R$ is a commutative ring and $\cM$ is the category of finitely generated modules, then declaring weak equivalences to be isomorphisms and cofibrations to be monomorphisms makes $\cM$ into a Waldhausen category. The following is straightforward verification.

\begin{lemma}\label{lem:toWald}
Let $\sA$ be a Waldhausen category and $\sD$ a small category. The category of functors $\mathsf{Fun}(\sD,\sA)$ is a Waldhausen category where
\begin{itemize}
\item[(z)] The zero object $Z \in \mathsf{Fun}(\sD, \sA)$ is the constant functor to the distinguished zero object in $\sA$;
\item[(w)] A natural transformation $\eta \colon F \Rightarrow G$ is a weak equivalence if and only if for each $d \in \sD$, $\eta_d \colon F(d) \to G(d)$ is an isomorphism; and
\item[(c)] A natural transformation $\alpha \colon F \Rightarrow G$ is a cofibration if and only if for each $d \in \sD$, $\alpha_d \colon F(d) \to G(d)$ is a monomorphism.
\end{itemize}
\end{lemma}

Given a Waldhausen category $\sC$, there is an associated simplicial Waldhausen category denoted $S_\bullet \sC$ and the subcategory, $wS_\bullet \sC$, of weak equivalences. The \emph{K-theory spectrum of $\sC$} is defined to be the $\Omega$-spectrum whose $nth$ space is given by
\[
\KK(\sC)_n := \lvert w \underbrace{S_\bullet S_\bullet \dotsb S_\bullet}_{n \text{ iterates}} \sC \rvert,
\]
i.e., the realization of the subcategory of weak equivalences of the $n$-fold (degreewise) application of the $S_\bullet$ construction.

\begin{definition} Let $\sA, \sE,$ and $\sB$ be Waldhausen categories. A sequence of exact functors
\[
\sA \xrightarrow{i} \sE \xrightarrow{f} \sB
\]
is \emph{exact} if
\begin{enumerate}
\item The composition $f \circ i$ is the zero map to $\sB$;
\item The functor $i$ is fully faithful; and
\item The functor $f$ restricts to an equivalence between $\sE / \sA$ and $\sB$.\footnote{Here,  $\sE/\sA$ is the full subcategory of $\sE$ on objects $e$ such that for all $a \in \sA$ the hom set $\sE (i(a), e)$ is a point.} 
\end{enumerate}
A sequence, as above, is \emph{split} if there exist exact functors
\[
\sA \xleftarrow {j} \sE \xleftarrow{g} \sB
\]
that are adjoint to $i$ and $f$ and such that the unit of the adjunction, $\mathrm{Id}_\sA \Rightarrow j \circ i$, and the counit of the adjunction, $f \circ g \Rightarrow \mathrm{Id}_\sB$, are natural isomorphisms.
\end{definition}

\begin{definition}\label{defn:standard}
 A split short exact sequence of Waldhausen categories
\[
\xymatrix{
\sA \ar[r]_{i} &
\sE \ar[r]_{f} \ar@/_1.0pc/[l]_{j} &
\sB \ar@/_1.0pc/[l]_{g},
}
\]
is \emph{standard} if 
\begin{enumerate}
\item For each $e \in \sE$, the component of the counit, $(i \circ j)(e) \to e$, is a cofibration;
\item For each cofibration $e \hookrightarrow e'$ in $\sE$, the induced map
\[
e \amalg_{(i \circ j)(e)} (i \circ j) (e') \to e'
\]
is a cofibration; and
\item If $a \to a' \to 0$ is a cofiber sequence in $\sA$, then the first map is an isomorphism.
\end{enumerate}
\end{definition}

The following is one of the fundamental theorems of algebraic K-theory. It is known as {\it Waldhausen Additivity}.

\begin{theorem}\label{thm:additivity}
Let
\[
\xymatrix{
\sA \ar[r]_{i} &
\sE \ar[r]_{f} \ar@/_1.0pc/[l]_{j} &
\sB \ar@/_1.0pc/[l]_{g},
}
\]
be a standard split SES of Waldhausen categories. Then the functors $i$ and $g$ induce an equivalence of spectra
\[
\KK(i) \vee \KK(g) \colon \KK(\sA) \vee \KK(\sB) \xrightarrow{\sim} \KK(\sE).
\]
\end{theorem}

\section{Proof of Theorem \ref{thm:equivcat}}\label{app:B}

The key ideas we use in the proof of the theorem go back to Grothendieck (and Verdier), specifically SGA4 Expos\'e VI \cite{SGA4}. The key observation---which we make precise---is that the category of zig--zag persistence modules is a localization of the \emph{Grothendieck construction} on the (psuedo)functor that sends a poset to its category of representations: $\cR \colon \mathsf{Poset}^{op} \to \mathsf{Cat}$. As we will explain, the Grothendieck construction is the lax colimit of $\cR$. Our domain category has an initial object, $\cZ \cZ_\NN$, hence, the colimit of $\Phi$ is isomorphic to the evaluation $\Phi (\cZ \cZ_\NN)$. Finally,  in Lemma \ref{lem:entrance} we recognized $\cZ \cZ_\NN$ as the poset underlying the (combinatorial) entrance path category of $\RR$ stratified with respect to the subset of natural numbers.

Throughout this appendix we will work with bicategories. Recall that any category can be considered as a bicategory with the only 2-morphisms being identities. The bicategory of categories, $\mathsf{Cat}$, consists of (small) categories, functors, and natural transformations. A reader who finds this appendix terse may find the recent book of Johnson and Yau \cite{JY} of great use.

\subsection{Psuedo and lax (co)limits}

When going from 1-categories to 2-categories there is more flexibility in definitions. This is already apparent when considering the notion of 2-functor and extends to limits and colimits as well. Many details of (co)limits in 2-categories were explicated in the 1980's by Ross Street and collaborators, for instance \cite{flex}. As an orienting exercise, let us recall the definition of lax and pseudo functors.

\begin{definition}
Let $\cA$ and $\cB$ be bicategories. A \emph{lax functor} $P \colon \cA \to \cB$ consists of
\begin{itemize}
\item A function $P \colon \mathrm{Obj}(\cA) \to  \mathrm{Obj}(\cB)$;
\item For each hom-category $\cA(X,Y)$ in $\cA$, a functor 
\[
P_{X,Y} \colon \cA(X,Y)\to \cB(P_X,P_Y);
\]
\item For each object $X\in \cA$ a 2-cell $P_{\id_X} \colon \id_{P_X} \Rightarrow P_{X,X}(1_X)$;
\item For each triple of objects and morphisms $f \colon X \to Y$ and $g \colon Y \to Z$ , a natural (in $f$ and $g$) transformation 
\[
P_{f,g} \colon P_{Y,Z} (g) \circ P_{X,Y}(f)  \Rightarrow   P_{X,Z}(g \circ f).
\] 
\end{itemize}
This data satisfies a sequence of coherence diagrams specifying unity and associativity.
\end{definition}

A lax functor is a \emph{psuedofunctor} if the 2-cells/natural transformations in the definition above are invertible. So a pseudofunctor is more strict than a lax functor, but not yet a strict 2-functor, which would require all higher morphisms to be identities. Correspondingly we have variable notions of colimit. For details see Chapter 5 of \cite{JY} and/or \cite{flex}.

\begin{definition}
Let $\Phi \colon \cA \to \cB$ be a lax functor.
\begin{itemize}
\item A \emph{lax colimit} of $\Phi$ is an initial object in the category of lax cocones under $\Phi$;
\item A \emph{psuedocolimit} of $\Phi$ is an initial object in the category of psuedococones under $\Phi$.
\end{itemize}
\end{definition}

The lax colimit of $\Phi$ is unique up to equivalence, while the psuedocolimit is unique up to isomorphism. We will use the notation $\colim \Phi$ for ``the" psuedocolimit of $\Phi$.

\begin{lemma}\label{lem:initial}
Let $\Phi \colon \cA \to \cB$ be a lax functor, $\cA$ an honest 1-category and $\fT \in \cA$ a terminal object. Then, we have an isomorphism
\[
\colim \Phi \cong \Phi (\fT).
\]
Correspondingly, if $\Psi \colon \cA^{op} \to \cB$ is a lax functor, $\cA$ an honest 1-category and $\fI \in \cA$ is initial, then
\[
\colim \Psi \cong \Psi (\fI).
\]
\end{lemma}

\subsection{The Grothendieck Construction}

\begin{definition}
Let $\mathsf{C}$ a category and $\Phi \colon \mathsf{C}^{op} \to \mathsf{Cat}$ be a lax functor. The \emph{Grothendieck construction}, $\int \Phi$, is the following category:
\begin{itemize}
\item An object of $\int \Phi$ is a pair, $(A,X)$, with $A \in \mathsf{C}$ and $X \in \Phi(A)$;
\item A morphism $(f,p) \colon (A, X) \to (B, Y)$ consists of
\begin{itemize}
\item A morphism $f \colon A \to B$ in the category $\mathsf{C}$; and
\item A morphism $p \colon X \to \Phi (f) (Y)$ in $\Phi (A)$.
\end{itemize}
\end{itemize}
\end{definition}

There are (reasonably) clear composition and identities in $\int \Phi$ and it is standard to verify that $\int \Phi$ actually defines a category.
The symbol ``$\int$" is meant to convey that the Grothendieck construction is amalgamating the data of $\Phi$ ``over" the domain category $\mathsf{C}$. Indeed, projection defines a functor $\int \Phi \to \mathsf{C}$ that is a fibration.

\begin{proposition}[Theorem 10.2.3 of \cite{JY}]
Let $\mathsf{C}$ a category and $\Phi \colon \mathsf{C}^{op} \to \mathsf{Cat}$ be a lax functor. The Grothendieck construction, $\int \Phi$, is a lax colimit of $\Phi$.
\end{proposition}

Let $U \colon \mathsf{E} \to \mathsf{C}$ be a functor and $\varphi \colon e \to e'$ a morphism in $\mathsf{E}$. Recall that $\varphi$ is \emph{cartesian} if every commutative triangle in $\mathsf{C}$ involving $U(\varphi)$ with a chosen lift of a 2-horn has a unique filler. (This definition is a bit colloquial, see Section 9.1 of \cite{JY}.)

\begin{corollary}
After localizing $\int \Phi$ at the collection of cartesian morphisms (with respect to projection $\int \Phi \to \mathsf{C}$) we obtain a pseudocolimit of $\Phi$, i.e., if $\mathrm{Cart}$ denotes the class of cartesian morphisms in $\int \Phi$, then
$\int \Phi [\mathrm{Cart}^{-1}] \cong \colim \Phi.$
\end{corollary}

\subsection{Proving the Theorem}

Let $\cR \colon \mathsf{Poset}^{\cZ\cZ_\NN/}_I \to \mathsf{Cat}$ be the psuedofunctor of linear representations, i.e., 
\[
\cR (\cZ\cZ_\NN \twoheadrightarrow \cP) := \mathsf{Fun} (\cP, \Vect).
\]
By design, the Grothendieck construction of $\cR$ recovers the category of marked zig-zag modules.

\begin{lemma} For $\cR$ defined above, $\int \cR \cong \mathsf{ZZmod}$.
\end{lemma}

\begin{lemma} A morphism $(f,\varphi) \colon (\cZ\cZ_\NN \twoheadrightarrow \cP , \rho) \to (\cZ\cZ_\NN \twoheadrightarrow \cQ , \eta)$ in $\mathsf{ZZmod}$ is cartesian if and only if $\varphi$ is a natural isomorphism.
\end{lemma}

\begin{proof}
Let $(f,\varphi) \colon (\cZ\cZ_\NN \twoheadrightarrow \cP , \rho) \to (\cZ\cZ_\NN \twoheadrightarrow \cQ , \eta)$ be a morphism and $(g,\psi) \colon (\cZ\cZ_\NN \twoheadrightarrow \cR , \alpha) \to (\cZ\cZ_\NN \twoheadrightarrow \cQ , \eta)$ a morphism such that $h \colon (\cZ\cZ_\NN \twoheadrightarrow \cR) \to (\cZ\cZ_\NN \twoheadrightarrow \cP)$ defines a commutative triangle in $\mathsf{Poset}^{\cZ\cZ_\NN /}_I$. We need to find a (unique) natural transformation $\chi \colon \alpha \Rightarrow h^\ast \rho$ such that $(h, \chi)$ fills the 2-horn upstairs in $\mathsf{ZZmod}$. This is possible precisely when $\varphi \colon \rho \Rightarrow \eta$ is an isomorphism. Indeed,  $g^\ast = h^\ast \circ f^\ast$, and $\psi \colon \alpha \Rightarrow g^\ast \eta$, so if $\varphi$ is an isomorphism we define
\[
\chi := \psi \colon \alpha \Rightarrow g^\ast \eta \cong h^\ast (\phi^\ast \eta) \cong h^\ast \rho.
\]
\end{proof}

\begin{lemma}
The object $(\cZ \cZ_\NN \xrightarrow{\; \mathrm{Id} \;} \cZ \cZ_\NN)$ is initial in $\mathsf{Poset}^{\cZ\cZ_\NN /}_I$.
\end{lemma}

\begin{proof}
Let $(\varphi \colon \cZ \cZ_\NN \twoheadrightarrow \cP) \in \mathsf{Poset}^{\cZ\cZ_\NN /}_I$. A map in the under category from $(\cZ \cZ_\NN \xrightarrow{\; \mathrm{Id} \;} \cZ \cZ_\NN)$ is a commutative diagram in $\mathsf{Poset}_I$
\[
\xymatrix{ &\cZ\cZ_\NN \ar[dl]_{\mathrm{Id}} \ar[dr]^{\mathrm{\varphi}} & \\
\cZ\cZ_\NN \ar[rr]_{\psi} && \cP.}
\]
By commutativity of the triangle, the map $\psi = \varphi$, so there is indeed a unique map in the under category.
\end{proof}

The precedings lemmas assemble to a proof of Theorem \ref{thm:equivcat}. More precisely, we have shown that
\[
\mathsf{ZZmod}[\cW^{-1}] \cong \smallint \cR [\mathrm{Cart}^{-1}] \cong \colim \cR \cong \cR (\cZ \cZ_\NN \xrightarrow{\; \mathrm{Id} \;} \cZ \cZ_\NN) \cong \mathsf{Fun} (\cZ \cZ_\NN , \Vect).
\]

\bibliographystyle{plain}
\bibliography{refs}

\begin{thebibliography}{10}

\bibitem{SGA4}
{\em Th\'{e}orie des topos et cohomologie {\'{E}}tale des sch\'{e}mas. {T}ome
  2}.
\newblock Lecture Notes in Mathematics, Vol. 270. Springer-Verlag, Berlin-New
  York, 1972.
\newblock S\'{e}minaire de G\'{e}om\'{e}trie Alg\'{e}brique du Bois-Marie
  1963--1964 (SGA 4), Dirig\'{e} par M. Artin, A. Grothendieck et J. L.
  Verdier. Avec la collaboration de N. Bourbaki, P. Deligne et B. Saint-Donat.

\bibitem{AFT}
David Ayala, John Francis, and Hiro~Lee Tanaka.
\newblock Local structures on stratified spaces.
\newblock {\em Advances in Mathematics}, 307:903--1028, 2017.

\bibitem{Barwick}
Clark Barwick, Saul Glasman, and Peter Haine.
\newblock Exodromy.
\newblock {\em arXiv preprint arXiv:1807.03281}, 2018.

\bibitem{belton2020reconstructing}
Robin~Lynne Belton, Brittany~Terese Fasy, Rostik Mertz, Samuel Micka, David~L
  Millman, Daniel Salinas, Anna Schenfisch, Jordan Schupbach, and Lucia
  Williams.
\newblock Reconstructing embedded graphs from persistence diagrams.
\newblock {\em Computational Geometry}, 90:101658, 2020.

\bibitem{BGO}
Nicolas Berkouk, Gr\'{e}gory Ginot, and Steve Oudot.
\newblock Level-sets persistence and sheaf theory.
\newblock {\em arXiv preprint arXiv:1907.09759}, 2019.

\bibitem{flex}
G.~J. Bird, G.~M. Kelly, A.~J. Power, and R.~H. Street.
\newblock Flexible limits for {$2$}-categories.
\newblock {\em J. Pure Appl. Algebra}, 61(1):1--27, 1989.

\bibitem{BGT}
Andrew~J. Blumberg, David Gepner, and Gon\c{c}alo Tabuada.
\newblock A universal characterization of higher algebraic {$K$}-theory.
\newblock {\em Geom. Topol.}, 17(2):733--838, 2013.

\bibitem{BubVirtual}
Peter Bubenik and Alex Elchesen.
\newblock Virtual persistence diagrams, signed measures, and {W}asserstein
  distance.
\newblock {\em arXiv preprint arXiv:2012.10514}, 2020.

\bibitem{BubHom}
Peter Bubenik and Nikola Mili{\'c}evi{\'c}.
\newblock Homological algebra for persistence modules.
\newblock {\em Foundations of Computational Mathematics}, pages 1--46, 2021.

\bibitem{carlsson2010zigzag}
Gunnar Carlsson and Vin de~Silva.
\newblock Zigzag persistence.
\newblock {\em Found. Comput. Math.}, 10(4):367--405, 2010.

\bibitem{chazal2013bootstrap}
Fr{\'e}d{\'e}ric Chazal, Brittany~Terese Fasy, Fabrizio Lecci, Alessandro
  Rinaldo, Aarti Singh, and Larry Wasserman.
\newblock On the bootstrap for persistence diagrams and landscapes.
\newblock {\em arXiv preprint arXiv:1311.0376}, 2013.

\bibitem{CP}
Justin Curry and Amit Patel.
\newblock Classification of constructible cosheaves.
\newblock {\em Theory Appl. Categ.}, 35:Paper No. 27, 1012--1047, 2020.

\bibitem{curry2015}
Justin~Michael Curry.
\newblock Topological data analysis and cosheaves.
\newblock {\em Jpn. J. Ind. Appl. Math.}, 32(2):333--371, 2015.

\bibitem{fiore}
Thomas~M. Fiore and Malte Pieper.
\newblock Waldhausen additivity: classical and quasicategorical.
\newblock {\em J. Homotopy Relat. Struct.}, 14(1):109--197, 2019.

\bibitem{GS}
Ryan Grady and Anna Schenfisch.
\newblock Natural stratifications of {R}eeb spaces and higher {M}orse
  functions.
\newblock {\em arXiv preprint arXiv:2011.08404}, 2020.

\bibitem{hatcher}
Allen Hatcher.
\newblock Algebraic topology, cambridge univ.
\newblock {\em Press, Cambridge}, 2002.

\bibitem{JY}
Niles Johnson and Donald Yau.
\newblock {\em 2-dimensional categories}.
\newblock Oxford University Press, Oxford, 2021.

\bibitem{KS}
Masaki Kashiwara and Pierre Schapira.
\newblock {\em Sheaves on manifolds}, volume 292 of {\em Grundlehren der
  mathematischen Wissenschaften [Fundamental Principles of Mathematical
  Sciences]}.
\newblock Springer-Verlag, Berlin, 1994.
\newblock With a chapter in French by Christian Houzel, Corrected reprint of
  the 1990 original.

\bibitem{lurie}
J.~Lurie.
\newblock Higher algebra.
\newblock available at
  \href{https://www.math.ias.edu/~lurie/papers/HA.pdf}{Author's Homepage}.

\bibitem{PatelGen}
Amit Patel.
\newblock Generalized persistence diagrams.
\newblock {\em J. Appl. Comput. Topol.}, 1(3-4):397--419, 2018.

\bibitem{Rosenberg}
Jonathan Rosenberg.
\newblock {\em Algebraic {$K$}-theory and its applications}, volume 147 of {\em
  Graduate Texts in Mathematics}.
\newblock Springer-Verlag, New York, 1994.

\bibitem{RSPL}
C.~P. Rourke and B.~J. Sanderson.
\newblock {\em Introduction to piecewise-linear topology}.
\newblock Springer-Verlag, New York-Heidelberg, 1972.
\newblock Ergebnisse der Mathematik und ihrer Grenzgebiete, Band 69.

\bibitem{thurston}
William~P. Thurston.
\newblock {\em Three-dimensional geometry and topology. {V}ol. 1}, volume~35 of
  {\em Princeton Mathematical Series}.
\newblock Princeton University Press, Princeton, NJ, 1997.
\newblock Edited by Silvio Levy.

\bibitem{Treumann}
David Treumann.
\newblock Exit paths and constructible stacks.
\newblock {\em Compos. Math.}, 145(6):1504--1532, 2009.

\bibitem{turner}
Katharine Turner, Sayan Mukherjee, and Doug~M Boyer.
\newblock Persistent homology transform for modeling shapes and surfaces.
\newblock {\em Information and Inference: A Journal of the IMA}, 3(4):310--344,
  2014.

\bibitem{Wald}
Friedhelm Waldhausen.
\newblock Algebraic {$K$}-theory of spaces.
\newblock In {\em Algebraic and geometric topology ({N}ew {B}runswick,
  {N}.{J}., 1983)}, volume 1126 of {\em Lecture Notes in Math.}, pages
  318--419. Springer, Berlin, 1985.

\bibitem{Weibel}
Charles~A. Weibel.
\newblock {\em The {$K$}-book}, volume 145 of {\em Graduate Studies in
  Mathematics}.
\newblock American Mathematical Society, Providence, RI, 2013.
\newblock An introduction to algebraic $K$-theory.

\bibitem{Inna1}
Inna Zakharevich.
\newblock The {$K$}-theory of assemblers.
\newblock {\em Adv. Math.}, 304:1176--1218, 2017.

\bibitem{Inna2}
Inna Zakharevich.
\newblock On {$K_1$} of an assembler.
\newblock {\em J. Pure Appl. Algebra}, 221(7):1867--1898, 2017.

\bibitem{zomorodian2005computing}
Afra Zomorodian and Gunnar Carlsson.
\newblock Computing persistent homology.
\newblock {\em Discrete \& Computational Geometry}, 33(2):249--274, 2005.

\end{thebibliography}

\end{document}